\documentclass[11pt]{article} 

\usepackage[utf8]{inputenc} 
\usepackage[british]{babel} 

\begin{hyphenrules}{british}
\end{hyphenrules}

\usepackage[margin=1.25in,footskip=0.25in]{geometry}
\usepackage{graphicx} 

\usepackage[parfill]{parskip} 

\usepackage{hyperref} 
\usepackage{booktabs} 
\usepackage{array} 
\usepackage{paralist} 
\usepackage{verbatim} 
\usepackage{subfig} 
\usepackage{bmpsize}
\usepackage{amssymb}
\usepackage{mathtools}
\usepackage{amsmath}	
\usepackage{amsthm}
\usepackage{graphicx}
\usepackage{faktor} 
\usepackage{mathrsfs} 
\usepackage[all, cmtip]{xy} 
\usepackage{multicol} 
\usepackage{diagbox} 
\usepackage{multirow} 
\usepackage{tikz} 
\usepackage{tikz-cd} 
\usetikzlibrary{decorations.pathreplacing}
\usepackage{array, makecell, rotating}
\usepackage{longtable} 
\usepackage{adjustbox} 
\usepackage{arydshln} 
\usepackage{halloweenmath} 
\usepackage{array, makecell, rotating}
\usepackage{longtable} 
\setcellgapes{2pt}
\newcommand\nocell[1]{\multicolumn{#1}{c|}{}}

\makeatletter
\renewcommand*\env@matrix[1][\arraystretch]{%
  \edef\arraystretch{#1}%
  \hskip -\arraycolsep
  \let\@ifnextchar\new@ifnextchar
  \array{*\c@MaxMatrixCols c}}
\makeatother  

\numberwithin{equation}{subsection}

\newcommand\Tstrut{\rule{0pt}{3.5ex}}       
\newcommand\Bstrut{\rule[-1.4ex]{0pt}{0pt}} 
\newcommand\TBstrut{\Tstrut\Bstrut}         
\newcommand\TTstrut{\rule{0pt}{6ex}}       
\newcommand\ptwiddle[1]{\mathord{\mathop{#1}\limits^{\scriptscriptstyle(\sim)}}}
\allowdisplaybreaks

\makeatletter
\newcommand{\subjclass}[2][2020]{%
  \let\@oldtitle\@title%
  \gdef\@title{\@oldtitle\footnotetext{#1 \textsc{MSC:} #2.}}%
}
\newcommand{\keywords}[1]{%
  \let\@@oldtitle\@title%
  \gdef\@title{\@@oldtitle\footnotetext{\emph{Keywords:} #1.}}%
}
\makeatother

\newcommand{\Address}{{
  \bigskip
  \footnotesize
  \textit{E-mail address}: \texttt{benedetta.piroddi@unimi.it} \textit{,\ also} \texttt{piroddi.benedetta@gmail.com}
\par\nopagebreak ORCID: 0000-0002-5327-2857
}}

\newtheorem{theorem}{Theorem}[subsection]

\theoremstyle{remark}
\newtheorem{definition}[theorem]{Definition}
\theoremstyle{theorem}
\newtheorem{proposition}[theorem]{Proposition}
\newtheorem{lemma}[theorem]{Lemma}
\newtheorem{corollary}[theorem]{Corollary}
\theoremstyle{remark}
\newtheorem{remark}[theorem]{Remark}
\theoremstyle{remark}

\usepackage{fancyhdr} 
\usepackage{sectsty}
\allsectionsfont{\sffamily\mdseries\upshape} 

\usepackage[nottoc,notlof,notlot]{tocbibind} 
\usepackage[titles,subfigure]{tocloft} 

\title{Symplectic actions of groups of order 4 on K3$^{[2]}$-type manifolds, and standard involutions on Nikulin-type orbifolds}
\date{}
\author{Benedetta Piroddi}

\begin{document}
\maketitle
\Address
\begin{abstract}
Given a K3$^{[2]}$-type manifold $X$ with a symplectic involution $i$, the quotient $X/i$ admits a Nikulin orbifold $Y$ as terminalization. We study the symplectic action of a group $G$ of order 4 on $X$, such that $i\in G$, and the natural involution induced on $Y$ (the two groups give two different results). We give a lattice-theoretic classification of $X$ and $Y$ in the projective case, and give some explicit examples of models of $X$. We also give lattice-theoretic criteria that a Nikulin-type orbifold $N$ has to satisfy to admit a symplectic involution that deforms to an induced one.
\end{abstract}


\section*{Aknowledgements}
This paper presents part of the results of the author's PhD thesis. She thanks the PhD program of Università degli Studi di Milano for the support: she is especially grateful to her advisor Alice Garbagnati, and to Bert van Geemen. She thanks Stevell Muller and Giovanni Bocchi for their invaluable help in accessing and using the database \cite{BHdata}.\\
She also thanks the anonymous referee of a previous version of this paper, for their careful reading, their corrections (especially those pertaining to Section \ref{sec:Nik}) and their earnest opinions: she thinks the implementation of their suggestions results in a better written paper overall.\\
The author is a member of INdAM GNSAGA.

\section*{Introduction} 
The Torelli theorem for hyperk\"ahler manifolds \cite{Markman} allows to deduce geometrical information about an hyperk\"ahler manifold from the interplay between the lattice and Hodge structure on its second integral cohomology group. Aiming to extend to the setting of (singular) symplectic varieties the framework of hyperk\"ahler geometry, cohomological information is especially welcome. Indeed, for some classes of symplectic varieties, Torelli-type theorems hold \cite{Menet3}, \cite{MR1}, \cite{BakkerLehn2021}; however, there are very few deformations type of symplectic varieties of which we know the lattice structure of the second integral cohomology group: one of them is the class of Nikulin-type orbifolds \cite{Menet2}. \\
A Nikulin orbifold $Y$ is obtained as terminalization of the quotient of a K3$^{[2]}$-type manifold $X$ by a symplectic involution $i$: if a bigger group $G$ in which $\langle i\rangle$ is normal acts on $X$, then $Y$ admits a residual action of $G/\langle i\rangle$
. If $G$ has order 4 the pair $(X,G)$ can be always deformed to a natural pair $(S^{[2]}, G)$, where $S$ is a K3 surface with a symplectic action of $G$ \cite{HM}; moreover, the symplectic action of $G$ on $H^2(S,\mathbb Z)$ is unique \cite[Thm. 4.7]{Nikulin2}, and completely described (see \cite{P} for $G=\mathbb Z/4\mathbb Z$, \cite{P2} for $G=(\mathbb Z/2\mathbb Z)^2$).\\
Our main result in this paper is the cohomological description of the two involutions induced on $Y$ by the symplectic action on $X$ of groups of order 4. To our knowledge, the only other automorphism ever described on $H^2(Y,\mathbb Z)$ at the time of writing is a non-standard symplectic involution \cite{MR2}. 

Proceeding as in \cite{CGKK}, we firstly describe in a lattice theoretic way the general member of families of projective K3$^{[2]}$-type manifolds $X$ with a symplectic action of $G$, (see Theorems \ref{familieshyperkahlerZ4} and \ref{familieshyperkahlerKlein}): these are classified by the pair $(NS(X),T(X))$ of the Néron-Severi and transcendental lattices of the general member.
For each group $G$, of two of these families we can actually give an explicit projective model of the general member: one as Fano manifold over a cubic fourfold, one as Hilbert scheme of two points of a quartic surface with a mixed (partially non-symplectic) action of $G$.

We then turn our attention to Nikulin orbifolds $Y$ and their deformation class. The two groups of order four induce two very different involutions $\iota_1,\iota_2$ on $Y$: indeed, we can see from the action on $X$ that the one induced by $\mathbb Z/4\mathbb Z$ fixes only points on $Y$, while the fixed locus of the one induced by $(\mathbb Z/2\mathbb Z)^2$ has codimension 2. We describe the action of these involutions on $H^2(Y,\mathbb Z)$ and deduce some information about the cohomology of the terminalization $W$ of the quotient $Y/\langle\iota_k\rangle$, extending to the case of K3$^{[2]}$-type manifolds the quotient maps we introduced for K3 surfaces in \cite{P}, \cite{P2}. After the recent proof of maximality of the monodromy group for Nikulin-type orbifolds \cite{BMM}, it follows that induced involutions can be deformed together with $Y$ to any Nikulin-type orbifold satisfying the appropriate lattice-theoretic conditions.

We conclude this paper with the lattice-theoretic classification of projective Nikulin orbifolds $Y$ that are terminalization of $X/i$, where $X$ is a general projective K3$^{[2]}$-type manifold with a symplectic action of a group of order 4 $G$, and $i\in G$ is an element of order 2: thus, we describe the correspondence between the moduli space of $X$ and $Y$ (this is a specialization of the correspondence in \cite{CGKK}, where $X$ is a general projective K3$^{[2]}$-type manifold with a symplectic involution). As standard involutions on Nikulin-type orbifolds commute with the non-standard involution $\kappa$ described in \cite{MR2}, we classify also projective Nikulin-type orbifolds that admit a mixed action of $(\mathbb Z/2\mathbb Z)^2$, where one of the generators is standard, and the other is $\kappa$. 

\section{Projective families of K3$^{[2]}$-type manifolds with a symplectic action of a group of order 4}

Let $(X,\varphi)$ be a marked hyperk\"ahler manifold, where $\varphi: (H^2(X,\mathbb Z),q_X)\simeq\Lambda$, $q_X$ being the BBF-form of $X$. Let a finite group $G$ act symplectically on $X$, and consider the invariant lattice $\Lambda^G\subset\Lambda$ and its orthogonal complement, the co-invariant lattice $\Omega_G$. \\
In general, $G$ can admit numerous different symplectic actions on hyperk\"ahler manifolds of the deformation type of $X$; its action in cohomology is however determined by the isometry class of $\Omega_G$, and the (primitive) embedding $\Omega_G\hookrightarrow\Lambda$. \\
Since $G$ preserves the symplectic form of $X$, its transcendental lattice $T(X)$ is invariant, and $\Omega_G$ is primitively embedded in $NS(X)=H^{1,1}(X)\cap H^2(X,\mathbb Z)$: since $\Omega_G$ is negative-definite, the general hyperk\"ahler manifold $X$ admitting a symplectic action of $G$ is not projective.

\begin{definition}[see \protect{\cite[\S 1]{Dolgachev}}]\label{def:latticepolarized}
Let $M$ be a negative-definite or hyperbolic even lattice admitting an embedding $j:M\hookrightarrow \Lambda$. A marked hyperk\"ahler manifold $(X,\varphi)$ is \emph{$(M,j)$-polarized} if $j$ factorizes as $M\hookrightarrow NS(X)\hookrightarrow H^2(X,\mathbb Z)\simeq\Lambda$, where both embeddings are primitive, and the last isometry is $\varphi$. We'll write \emph{$M$-polarized} if $j$ is unique up to isometries of $\Lambda$.
\end{definition}

\begin{remark}\label{moduli_space_dim}
The moduli space of $(M,j)$-polarized marked hyperk\"ahler manifolds deformation equivalent to $(X,\varphi)$ has dimension $rk(\Lambda)-2-rk(M)$. 
\end{remark}

\begin{definition}\label{def:proj_families}
Let $G$ be a finite group acting symplectically, and fix the associated embedding $\psi: \Omega_G\hookrightarrow\Lambda$. Let $(S,j)$ be a pair such that $S$ is an even hyperbolic lattice of rank rk$\Omega_G+1$ and $\psi$ factorizes in primitive embeddings $j\circ\overline\psi:\Omega_G\hookrightarrow S\hookrightarrow\Lambda$. Denote each family of marked $(S,j)$-polarized hyperk\"ahler manifolds $(X,\varphi)$ a \emph{projective family (with an action of $G$)}.
\end{definition}

\begin{proposition} \label{prop:proj_families} 
A given symplectic action of a finite group $G$ admits at most countably many projective families. A projective hyperk\"ahler manifold $X$ admits the chosen action of $G$ if and only if it belongs to one such projective family. 
\end{proposition}
\begin{proof}
If $X$ is projective and admits a given symplectic action of $G$, then $NS(X)$ contains primitively both the negative definite lattice $\Omega_G$ and an ample class $L$ of square $2d$: therefore, $S=\Omega_G\oplus\langle 2d\rangle$ or an overlattice of finite index. For each suitable $d>0$ there exists a finite number of such overlattices \cite[Prop. 1.4.1]{Nikulin1}, and for each choice of $S$ there is a finite number of choices of embeddings $j:S\hookrightarrow\Lambda$  \cite[Prop. 1.15.1]{Nikulin1}.
Conversely, if $X$ belongs to a projective family, then it is projective \cite[Thm. 2]{Huybrechts}, and it admits a symplectic action of $G$ by the Torelli theorem for hyperk\"ahler manifolds \cite[Thm. 1.3]{Markman}.
\end{proof}

\begin{remark}\textbf{Overlattice notation.}
Consider the lattice $N\oplus\langle k\rangle$, where $N$ is a negative definite
even lattice and $\langle k\rangle$ is an even positive definite lattice with intersection matrix $[k]$.
Denote $(N\oplus\langle k\rangle)'$ and $(N\oplus\langle k\rangle)^\star$ any cyclic overlattices of $N\oplus\langle k\rangle$ obtained by adding to the list of generators a class of the form $(\nu + \kappa)/2, (\nu + \kappa)/4$ respectively, with $\nu\in N$ and $\kappa$ the generator of $\langle k\rangle$. When two overlattices of index 2 of  $N\oplus\langle k\rangle$ as above are not abstractly isomorphic, they will be denoted as $(N\oplus\langle k\rangle)'^{(i)},\ i = 1, 2$.
\end{remark}

\textbf{Notation:} Let $X$ be a K3$^{[2]}$-type manifold. It holds \[H^2(X,\mathbb Z)\simeq \Lambda_{\mathrm{K3}^{[2]}}=E_8^{\oplus 2}\oplus U^{\oplus 3}\oplus\langle -2\rangle\simeq\Lambda_{\mathrm{K3}}\oplus\mathbb Z\mu,\]
where $\mu$ is a class of square $-2$ and divisibility 2: notice that this class admits a unique primitive embedding in $\Lambda_{\mathrm{K3}^{[2]}}$ up to isometries, its orthogonal complement always being $E_8^{\oplus 2}\oplus U^{\oplus 3}\simeq\Lambda_{\mathrm{K3}}$. 
\begin{remark}
If $X=\Sigma^{[2]}$ is the Hilbert scheme of two points on a K3 surface $\Sigma$, then $\mu=\Delta/2$, where $\Delta$ is the exceptional divisor of the blow-up $\Sigma^{[2]}\rightarrow Sym^2(\Sigma)$.
\end{remark}

\begin{remark}\label{G_action_standard}
Finite symplectic actions on K3$^{[2]}$-type manifolds are classified in \cite{HM}. In particular, if $G$ has order 4 the action of $G$ on $\Lambda_{\mathrm{K3}^{[2]}}$ is unique (depeding only on $G$, up to isometries of $\Lambda_{\mathrm{K3}^{[2]}}$): the co-invariant lattice $\Omega_G$ is the same as the one for the action of $G$ on a K3 surface, and it is embedded in  $\Lambda_{\mathrm{K3}^{[2]}}$ orthogonally to $\mu$, and uniquely up to isometries of $\mu^\perp$.
\end{remark}

Denote $\Omega_4$ and $\Omega_{2,2}$  the co-invariant lattices for the action of $G=\mathbb Z/4\mathbb Z$ and $(\mathbb Z/2\mathbb Z)^2$ respectively, and fix the embedding $\Omega_G\hookrightarrow\Lambda_{\mathrm{K3}}\hookrightarrow \Lambda_{\mathrm{K3}}\oplus\mathbb Z\mu\simeq\Lambda_{\mathrm{K3}^{[2]}}$. The lattices $S$ satisfying Definition \ref{def:proj_families} are classified in \cite[Thm. 5.1.4]{P}, \cite[Thm. 3.1.3]{P2}, and each embedding $j:S\hookrightarrow\Lambda$ is determined by the embedding in $\Omega_G^\perp$ of the class of positive square that generates $\Omega_G^{\perp_{S}}$. We're going to distinguish two cases: 
\begin{itemize}[--]
\item the class that generates $\Omega_G^{\perp_{S}}$ is of the form $L=L_\Sigma\in\mu^\perp$: there exists a general projective K3 surface $\Sigma$ admitting a symplectic action of $G$, such that $S\simeq NS(\Sigma)$, $T\simeq T(\Sigma)\oplus\langle -2\rangle$; the classes $L_\Sigma$ are described (up to isometries of $\Lambda_{\mathrm{K3}}$) in \cite[Ex. 5.1.6]{P} for $G=\mathbb Z/4\mathbb Z$, in \cite[Table 2]{P2} for $G=(\mathbb Z/2\mathbb Z)^2$.
\item the class that generates $\Omega_G^{\perp_{S}}$ is of the form
$M=\lambda+n\mu$, where $\lambda\in\Lambda_{\mathrm{K3}}$ is a class of positive square, and $n>0$. 
\end{itemize}

\subsection[The case $G=\mathbb Z/4\mathbb Z$]{Families of projective K3$^{[2]}$-type manifolds with a symplectic action of $\mathbb Z/4\mathbb Z$}\label{K3[2]withZ4}

Let $G=\mathbb Z/4\mathbb Z=\langle\tau\rangle$. By Remark \ref{G_action_standard} it holds
\[\Lambda_{\mathrm{K3}^{[2]}}^{\tau}=\Lambda_{\mathrm{K3}}^{\tau}\oplus\langle -2\rangle, \]
where we fix the following notation for the generators:
\begin{equation}\label{marking_invariant_Z4}
\xymatrix@C=0.001pc{
{\begin{matrix}\Lambda_{\mathrm{K3}}^{\tau}\ \simeq \\ \ \end{matrix}} &{\begin{matrix}U \\ s_1, s_2\end{matrix}} &{\begin{matrix}\oplus\\ \ \end{matrix}} &{\begin{matrix}\langle -2\rangle^{\oplus 2} \\ w_1, w_2\end{matrix}} &{\begin{matrix}\oplus\\ \ \end{matrix}} &{\begin{matrix}U(4)^{\oplus 2} \\ w_3,\dots,  w_6\end{matrix}}}
\end{equation}

\begin{theorem}\label{familieshyperkahlerZ4}
The projective families of K3$^{[2]}$-type manifolds $X$ with a symplectic action of $\mathbb Z/4\mathbb Z$ correspond to the pairs $(S,T)$ appearing in the following table, where the class $L$ indicated is the generator of $\langle 2d\rangle=\Omega_4^{\perp_{S}}$, and $T=S^{\perp}$: the classes $L_0, L_{i,j}$ are defined in \cite[Ex. 5.1.6]{P}, while the $M_i$ and $\tilde M_i$ are defined in the proof below.
{\footnotesize
\begin{longtable}{|c|c|c|c|c|}
\cline{2-5}
\nocell{1} &{$S$\Tstrut} &{$T$ \Tstrut} & $L$ &$L^2$\\ [6pt]
\hline
\multirow{4}{*}{$d=_4 1$\Tstrut}	&\multirow{3}{*}{$\Omega_4\oplus\langle 2d\rangle$\TTstrut} & $U(4)^{\oplus 2}\oplus\langle-2\rangle^{\oplus 3}\oplus\langle 2d\rangle$\Tstrut &$L_0(d)$ &$2d$\\[6pt]
\ &\ & $U(4)^{\oplus 2}\oplus K_m$ &$M_1(m)$ &$2(4m-3)$ \\[6pt]
\ &\ & ${U(4)\oplus D_m}$ &$\tilde M_1(m)$ &$2(4m+1)$ \\[6pt]
\hline
\multirow{5}{*}{$d=_4 2$\Tstrut} & \multirow{2}{*}{$\Omega_4\oplus\langle 2d\rangle$\Tstrut} &$U(4)^{\oplus 2}\oplus\langle-2\rangle^{\oplus 3}\oplus\langle 2d\rangle$\Tstrut &$L_0(d)$ &$2d$\\
\ &\ &{$U(4)^{\oplus 2}\oplus\langle-2\rangle\oplus H_m$\Tstrut} &$M_2(m)$ &$2(4m-2)$\\[6pt]
\cdashline{2-5}
\ &$(\Omega_4\oplus\langle 2d\rangle)'^{(1)}$\Tstrut &\multirow{2}{*}{$U(4)^{\oplus 2}\oplus\langle-2\rangle\oplus H_m$\Tstrut} &$L_{2,2}^{(1)}(m)$ &$2(4m+2)$ 
\\ 
\ & $(\Omega_4\oplus\langle 2d\rangle)'^{(2)}$\Tstrut  &\  &$L_{2,2}^{(2)}(m)$ &$2(4m-2)$\\ [6pt]
\hline
\end{longtable}\newpage
\begin{longtable}{|c|c|c|c|c|}
\cline{2-5}
\nocell{1} &{$S$\Tstrut} &{$T$ \Tstrut} & $L$ &$L^2$\\ [6pt]
\hline
\multirow{3}{*}{$d=_4 3$\Tstrut}&\multirow{2}{*}{$\Omega_4\oplus\langle 2d\rangle$\Tstrut} & $U(4)^{\oplus 2}\oplus\langle-2\rangle^{\oplus 3}\oplus\langle 2d\rangle$\Tstrut &$L_0(d)$ &$2d$\\[6pt]
\ &\ &$U(4)^{\oplus 2}\oplus\langle-2\rangle^{\oplus 2}\oplus G_m$ &$M_3(m)$ &$2(4m-1)$\\[6pt]
\cdashline{2-5}
\ &$(\Omega_4\oplus\langle 2d\rangle)'$\Tstrut &$U(4)^{\oplus 2}\oplus\langle-2\rangle^{\oplus 2}\oplus G_m$\Tstrut &$L_{2,3}(m)$ &$2(4m+3)$\\[6pt]
\hline
\multirow{6}{*}{$d=_4 0$\Tstrut} & $\Omega_4\oplus\langle 2d\rangle$\Tstrut &$U(4)^{\oplus 2}\oplus\langle-2\rangle^{\oplus 3}\oplus\langle 2d\rangle$\Tstrut  &$L_0(d)$ &$2d$   \\[6pt]
\cdashline{2-5}
\ &\multirow{2}{*}{$(\Omega_4\oplus\langle 2d\rangle)'$\Tstrut} & $U(4)\oplus\langle-2\rangle^{\oplus 3}\oplus F_m$\Tstrut  &$L_{2,0}(m)$ &$2(4(m-1))$\\[6pt]
\ &\ &$U\oplus U(4)\oplus\langle-2\rangle^{\oplus 2}\oplus E_m$\Tstrut &$M_4(m)$ &$2(4(m-1))$\\[6pt]
\cdashline{2-5}
\ & $(\Omega_4\oplus\langle 2d\rangle)^{\star}$\Tstrut  &$U\oplus U(4)\oplus\langle-2\rangle^{\oplus 2}\oplus E_m$\TBstrut  &$L_{4,j}(h)$ &$2(4(m-1))$, see Table \ref{relation_mjh}\\ [6pt]
\hline
\end{longtable}

\begin{center}\begin{table}[h!]
\caption{Relation between 
$m,j,h$}\label{relation_mjh}
\centering
\begin{tabular}{c|c c c c}
$m$ (mod 4) & 0 & 1 & 2 & 3 \\
\hline
$j$ & 12 & 0 & 4 & 8\\
$h$ & $(m-4)/4$ &$(m+3)/4$  &$(m-2)/4$ &$(m+13)/4$
\end{tabular}\end{table}\end{center}
\begin{gather*}
G_m=\left[\begin{array}{r r}-2m &1\\ 1&-2\end{array}\right] \quad
H_m=\left[\begin{array}{r r r}-2m &1 &1\\ 1&-2 &0\\1 & 0 & -2\end{array}\right]\quad
K_m=\left[\begin{array}{r r r r}-2m &1 &1 &0\\ 1&-2 &0 &0\\1 & 0 & -2 &2\\0 & 0 & 2 & -4\end{array}\right]\\
D_m=\left[\begin{array}{r r r r r r} -2m &1 &1 &1 &2 &2\\
     1 &-2 &0 &0 &0 &0\\
     1 &0 &-2 &0 &0 &0\\
     1 &0 &0 &-2 &0 &0\\
     2 &0 &0 &0 &0 &4\\
     2 &0 &0 &0 &4 &0\end{array}\right]\quad
E_m=\left[\begin{array}{r r}-8m & 4\\ 4 &-2\end{array}\right]\quad
F_m=\left[\begin{array}{r r r}-2(m-1) & 2 & 0\\ 2 &0 &4\\ 0 &4 &0\end{array}\right]
\end{gather*}}
\normalsize
\end{theorem}
\begin{proof}
We use \cite[Prop. 1.15.1]{Nikulin1} to determine all primitive embeddings $j:S\hookrightarrow\Lambda_{\mathrm{K3^{[2]}}}$ of each of the lattices $S$ satisfying Definition \ref{def:proj_families}: we find that we can distinguish each $j$ by the orthogonal complement to $j(S)$, that is the corresponding lattice $T$ appearing in the table.\\
The discriminant group  of the ambient lattice $\Lambda_{\mathrm{K3^{[2]}}}$ is $\mathbb Z/2\mathbb Z$ with discriminant form $q=[3/2]$. Since we've fixed the embedding $\Omega_4\hookrightarrow\Lambda_{\mathrm{K3}}\hookrightarrow \Lambda_{\mathrm{K3}}\oplus\mathbb Z\mu$, we can see that for each $S$ we always have at least the primitive embedding orthogonal to $\mathbb Z\mu$ (see \cite[Ex. 5.1.6]{P}); as necessary condition to have alternative embeddings, the discriminant form of $S$ has to contain an element of order 2 and square 3/2: if there is more than one such element, we check according to \cite[Prop. 1.15.1]{Nikulin1} if they give rise to different embeddings. 

The lattice $\Omega_4\oplus\langle2d\rangle$ has discriminant form
\[q_{\Omega_4\oplus\langle2d\rangle}:=\begin{bmatrix}0 &1/4\\1/4 &0\end{bmatrix}^{\oplus 2}\oplus\begin{bmatrix}1/2\end{bmatrix}^{\oplus 2}\oplus \begin{bmatrix}1/2d\end{bmatrix}:\]
let $\gamma$ be the generator of the subgroup $[1/2d]$, $\alpha_1,\alpha_2$ those of $[1/2]^{\oplus 2}$, $x_1,x_2$ those of one of the $ \begin{bmatrix}0 &1/4\\1/4 &0\end{bmatrix}$ blocks. For $d=_4 3$, $d\gamma$ has order 2 and square $3/2$; for $d=_4 1$, $d\gamma$ has order 2 and square $1/2$, so $\alpha_1+\alpha_2+d\gamma$ has order 2 and square $3/2$; for $d=_4 2$, $d\gamma$ has order 2 and square $1$, so $\alpha_1+d\gamma$ has order 2 and square $3/2$; for $d=_4 0$, $d\gamma$ has order 2 and square $0$, so we have no alternative embeddings. 
The alternative embeddings of $\Omega_4\oplus\langle2d\rangle$ are realized by the following classes of square $2d$:
\begin{itemize}
\item for $d=4m-3$, $M_1(m)=2(s_1+ms_2)+w_2-w_1+\mu$, %
\item for $d=4m-2$, $M_2(m)=2(s_1+ms_2)+w_1+\mu$;
\item for $d=4m-1$, $M_3(m)=2(s_1+ms_2)+\mu$.
\end{itemize}
For $d=_4 1$, the class
\[\tilde M_1(m)=2(s_1+ms_2)+w_3+w_4+w_2-w_1+\mu\]
provides a third different primitive embedding of $\Omega_4\oplus\langle2d\rangle$ in the ambient lattice: the associated subgroup of $q_{\Omega_4\oplus\langle2d\rangle}$ is generated by $\alpha_1+\alpha_2+d\gamma+x_1+x_2$.\\
For $d=_4 0$ the lattice $(\Omega_4\oplus\langle2d\rangle)'$ admits another primitive embedding, realized by the following classes:
\begin{itemize}
\item for $d=4(m-1)$, $M_4(m)=w_3+mw_4+2\mu$.
\end{itemize}
These are the only cases in which there exist alternative embeddings, as the discriminant group of the other Néron-Severi lattices does not contain any element of order 2 and square $3/2$.
\end{proof}

\subsection[The case $G=(\mathbb Z/2\mathbb Z)^2$]{Families of projective K3$^{[2]}$-type manifolds with a symplectic action of $(\mathbb Z/2\mathbb Z)^2$}\label{K3[2]withKlein}

Let $G=(\mathbb Z/2\mathbb Z)^2=\langle\tau,\varphi\rangle$. By Remark \ref{G_action_standard}, it holds
\[\Lambda_{\mathrm{K3}^{[2]}}^G=\Lambda_{\mathrm{K3}}^G\oplus\langle -2\rangle,\]
where
\begin{equation}\label{marking_invariant_Z22}
\xymatrix@C=0.001pc{
{\begin{matrix}\Lambda_{\mathrm{K3}}^G \simeq\\ \ \end{matrix}} &{\begin{matrix}U \\ s_1, s_2\end{matrix}} &{\begin{matrix}\oplus\\ \ \end{matrix}} &{\begin{matrix}U(2)^{\oplus 2} \\ u_1,\dots, u_4\\\end{matrix}} &{\begin{matrix}\oplus\\ \ \end{matrix}} &{\begin{matrix}D_4(2) \\ m_1,\dots,  m_4\end{matrix}}} \end{equation}

\begin{theorem}\label{familieshyperkahlerKlein}
The projective families of K3$^{[2]}$-type manifolds $X$ with a symplectic action of $(\mathbb Z/2\mathbb Z)^2$ correspond to the pairs $(S,T)$ appearing in the following table, where the class $L$ indicated is the generator of $\langle 2d\rangle=\Omega_{2,2}^{\perp_{S}}$, and $T=S^{\perp}$: the classes $L_0, L_{i,j}^{(h)}$ are defined in \cite[Table 2]{P2}, while the $M_i$ are defined in the proof below.
\end{theorem}
\footnotesize
\begin{longtable}{|c|c|c|c|c|}
\cline{2-5}
\nocell{1} &{$S$\Tstrut} &{$T$ \Tstrut} & $L$ &$L^2$\\ [6pt]
\hline
\multirow{2}{*}{$d=_4 1$\Tstrut}	&\multirow{2}{*}{$\Omega_{2,2}\oplus\langle 2d\rangle$\Tstrut} & $\langle-2d\rangle\oplus\langle-2\rangle\oplus U(2)^{\oplus 2}\oplus D_4(2)$\Tstrut &$L_0(d)$ &$2d$\\[6pt]
\ &\ & $U(2)^{\oplus 2}\oplus B_m$ &$M_1(m)$ &$2(4m-3)$ \\[6pt]
\hline
\multirow{2}{*}{$d=_4 3$\Tstrut}	&\multirow{2}{*}{$\Omega_{2,2}\oplus\langle 2d\rangle$\Tstrut} & $\langle-2d\rangle\oplus\langle-2\rangle\oplus U(2)^{\oplus 2}\oplus D_4(2)$\Tstrut &$L_0(d)$ &$2d$\\[6pt]
\ &\ & $U(2)^{\oplus 2}\oplus D_4(2)\oplus G_m$ &$M_3(m)$ &$2(4m-1)$ \\[6pt]
\hline
\multirow{3}{*}{$d=_4 2$\Tstrut}	&{$\Omega_{2,2}\oplus\langle 2d\rangle$\Tstrut} & $\langle-2d\rangle\oplus\langle-2\rangle\oplus U(2)^{\oplus 2}\oplus D_4(2)$\Tstrut &$L_0(d)$ &$2d$\\[6pt]
\ &$(\Omega_{2,2}\oplus\langle 2d\rangle)'$\Tstrut & $\langle-2\rangle\oplus D_4(2)\oplus P_h$ &$L_{2,2}^{(a,b)}(h)$ &$2(4h+2)$ \\[6pt]
\hline
\multirow{5}{*}{$d=_8 0$\Tstrut} &{$\Omega_{2,2}\oplus\langle 2d\rangle$\Tstrut} &$\langle-2d\rangle\oplus\langle-2\rangle\oplus U(2)^{\oplus 2}\oplus D_4(2)$\Tstrut &$L_0(d)$ &$2d$\\
\ &{$(\Omega_{2,2}\oplus\langle 2d\rangle)'^{(1)}$\Tstrut} &$\langle-2\rangle\oplus U\oplus R_h$\Tstrut &$L_{2,0}^{(1)}(h)$ &$2(4h)$\\[6pt]
\cdashline{2-5}
\ &\multirow{2}{*}{$(\Omega_{2,2}\oplus\langle 2d\rangle)'^{(2)}$\Tstrut} &{$\langle-2\rangle\oplus U(2)^{\oplus 2}\oplus Q_h$\Tstrut} &$L_{2,0}^{(2)}(h)$ &$2(4h-4)$ \\ 
\ &\ &$U\oplus U(2)\oplus C_m$\Tstrut   &$M_8(m)$ &$2(8m-8)$\\ [6pt]
\hline
\multirow{5}{*}{$d=_8 4$\Tstrut}&{$\Omega_{2,2}\oplus\langle 2d\rangle$\Tstrut} & $\langle-2d\rangle\oplus\langle-2\rangle\oplus U(2)^{\oplus 2}\oplus D_4(2)$\Tstrut &$L_0(d)$ &$2d$\\
\ &{$(\Omega_{2,2}\oplus\langle 2d\rangle)'^{(1)}$\Tstrut} &$\langle-2\rangle\oplus U\oplus R_h$\Tstrut &$L_{2,0}^{(1)}(h)$ &$2(4h)$\\
\ &{$(\Omega_{2,2}\oplus\langle 2d\rangle)'^{(2)}$\Tstrut} &{$\langle-2\rangle\oplus U(2)^{\oplus 2}\oplus Q_h$\Tstrut} &$L_{2,0}^{(2)}(h)$ &$2(4h-4)$ \\
\ &{$(\Omega_{2,2}\oplus\langle 2d\rangle)^\star$\Tstrut} &$\langle-2\rangle\oplus T^{(\pm4)}_h$ &$L_{4,\pm 4}(h)$ &$2(16h\pm 4)$\TBstrut\\ \hline
\end{longtable}
\scriptsize
\begin{gather*}
G_m=\left[\begin{array}{r r}-2m &1\\ 1&-2\end{array}\right] \quad 
B_m = \left[\begin{array}{ c | c }
    G_m & \begin{array}{r r r r}1 & 1 &-2 &1 \\0 & 0 & 0 & 0 \end{array}\\
    \hline
    \begin{array}{c c} 1 & 0 \\
  1 & 0 \\
 -2 & 0 \\
  1 & 0 \end{array} & D_4(2)
  \end{array}\right]
C_m = \left[\begin{array}{ r | c }
    \begin{array}{r r}-4m &2\\ 2 &-2\end{array} & \begin{array}{r r r r}-2 & 2 &0 &0 \\0 & 0 & 0 & 0\end{array}\\
    \hline
    \begin{array}{r r} -2 & 0 \\
  2 & 0 \\
  0 & 0 \\
  0 & 0 \end{array} & D_4(2)
  \end{array}\right]
\\[10pt]
P_h = \left[\begin{array}{ c | c }
    -2h & \begin{array}{r r r r} 1 & 1 &1 &0\end{array}\\
    \hline
    \begin{array}{r} 1 \\ 1\\ 1 \\ 0\end{array} & U(2)^{\oplus 2}
  \end{array}\right]\enspace
Q_h = \left[\begin{array}{ r | c }
    -2h & \begin{array}{r r r r} 2 & 0 &0 &0\end{array}\\
    \hline
    \begin{array}{r} 2 \\ 0\\ 0 \\ 0\end{array} & D_4(2)
  \end{array}\right]\enspace
T^{(-4)}_h= U(2)^{\oplus 2}\oplus
\left[\begin{array}{ r | c }
    -2h & \begin{array}{r r r r} 1 & 0 &0 &0\end{array}\\
    \hline
    \begin{array}{r} 1 \\ 0\\ 0 \\ 0\end{array} & D_4(2)
  \end{array}\right]\\[10pt]
T^{(4)}_h= U\oplus\left[\begin{array}{r r r r r r r }-8h &-2 &0 &0 &0 &0 &0\\
     -2 &0 &2 &0 &0 &0 &0\\
     0 &2 &-4 &-2 &0 &0 &0\\
     0 &0 &-2 &-4 &-4 &2 &0\\
     0 &0 &0 &-4 &-8 &4 &0\\
     0 &0 &0 &2 &4 &-4 &2\\
     0 &0 &0 &0 &0 &2 &-4\\\end{array}\right]
\enspace
R_h=\left[\begin{array}{r r r r r r r}
-8(h+1) &2 &0 &0 &0 &0 &0\\
     2 &0 &2 &0 &0 &0 &0\\
     0 &2 &4 &2 &0 &0 &0\\
     0 &0 &2 &-4 &2 &0 &0\\
     0 &0 &0 &2 &-4 &4 &0\\
     0 &0 &0 &0 &4 &-8 &4\\
     0 &0 &0 &0 &0 &4 &-8
\end{array}\right]\\
\end{gather*}
\normalsize
\begin{proof}
The proof is analogous to that of Theorem \ref{familieshyperkahlerZ4}. We now use the lattices $S$ described for K3 surfaces $\Sigma$ in \cite[Thm. 3.1.3]{P2}; again, fixing the embedding $\Omega_{2,2}\hookrightarrow\Lambda_{\mathrm{K3}}\hookrightarrow \Lambda_{\mathrm{K3}}\oplus\mathbb Z\mu\simeq\Lambda_{\mathrm{K3}^{[2]}}$, for each $S\simeq NS(\Sigma)$ we have at least the embedding such that $T\simeq T(\Sigma)\oplus\langle -2\rangle$ (see \cite[Table 2]{P2}). The additional choices are as follows: for $d$ odd, the lattice $\Omega_{2,2}\oplus\langle2d\rangle$ admits another primitive embedding, realized by the following classes of square $2d$ in $\Lambda_{\mathrm{K3}^{[2]}}^G=\Lambda_{\mathrm{K3}}^G\oplus\mathbb Z\mu$:
\begin{itemize}
\item for $d=4m-3$, $M_1(m)=2(s_1+ms_2)+m_3+\mu$;
\item for $d=4m-1$, $M_3(m)=2(s_1+ms_2)+\mu$.
\end{itemize}
For $d=_8 0$ the lattice $(\Omega_{2,2}\oplus\langle2d\rangle)'^{(2)}$ admits another primitive embedding, realized by
\begin{itemize}
\item $M_8(m)=2(\mu+u_3+mu_4)+m_2-m_1$.
\end{itemize}
These are the only cases in which there exist alternative embeddings, as the discriminant group of the other lattices $S$ does not contain any element of order 2 and square $3/2$.
\end{proof}

\subsection{Fixed loci}

\begin{remark}\label{fixed_loci_deformation}
Let $X$ be a K3$^{[2]}$-type manifold, and $G$ a finite group acting symplectically on $X$ and fixing a class $\mu$ of square $-2$ and divisibility 2. Then,
the locus of points of $X$ with non-trivial stabilizer will be diffeomorphic to that of the natural action of $G$ on $S^{[2]}$, where $S$ is a K3 surface admitting itself a symplectic action of $G$. Indeed, the fixed locus for any element of $G$ is smooth and consists of points, abelian and K3 surfaces \cite[Prop. 2.6]{Fujiki}, but smooth deformations of a K3 (resp. abelian) surface are still K3 (resp. abelian) surfaces, and smooth deformations of points are points. The number of points, K3 and abelian surfaces in the locus with non-trivial stabilizer under the action of $G$, and their intersections, are therefore preserved under deformations.
\end{remark}

\begin{proposition}\label{fixZ4onhyperkahler}
Let $G=\mathbb Z/4\mathbb Z$ act symplectically on a K3$^{[2]}$-type manifold $X$, let $\tau$ be a generator of $G$. Then $\tau$ fixes 16 points, of which exactly 8 lie on the K3 surface $\Sigma$ fixed by $\tau^2$.
\end{proposition}
\begin{proof}
By Remark \ref{fixed_loci_deformation} we can reduce to $X=S^{[2]}$, where $S$ is a K3 surface with a symplectic action of $G$. Call $\tau$ both the generator of the action of $G$ on $S$, and that of the natural action on $X$.\\
On $S$, $\tau$ fixes 4 points $\{p_1,\dots p_4\}$ and exchanges two pairs of points, $q_1\mapsto q_2,\ r_1\mapsto r_2$ (these are fixed by $\tau^2$) \cite[\S 5.2]{Nikulin2}. \\
Let $[s_1,s_2]$ with $s_1\neq s_2$ be the class in $S^{[2]}$ of the unordered pair of points $\{s_1,s_2\}\subset S$, and denote $\Delta$ the exceptional line bundle resulting from the blow-up of the singular locus of $Sym^2(S)$. The natural action of $\tau^2$ on $S^{[2]}$ fixes 28 isolated points: 6 of the form $[p_i,p_j]$ with $i\neq j\in\{1,\dots, 4\}$, 16 of the form $[p_i,q_j]$ or $[p_i,r_j]$ with $i\in\{1,\dots, 4\}, j\in\{1,2\}$, 4 of the form $[q_i,r_j],\ i,j\in\{1,2\}$, and the points $[q_1,q_2],[r_1,r_2]$; moreover, it fixes the K3 surface $\Sigma$, the closure of the set $\{[s,\tau^2(s)],\ s\in S\}$. The intersection $\Sigma\cap \Delta$ consists of 8 lines, over the points $[p_i,p_i],[q_j,q_j],[r_j,r_j],  i\in\{1,\dots, 4\}, j\in\{1,2\}$. The involution induced by the action of $G/\langle\tau^2\rangle$ on $\Sigma$ exchanges pairwise the four lines over $[q_j,q_j],[r_j,r_j]$, and fixes two points on each of the remaining four lines.
\end{proof}

\begin{proposition}\label{fixKleinonhyperkahler}
Let $G=(\mathbb Z/2\mathbb Z)^2$ act symplectically on a K3$^{[2]}$-type manifold $X$, let $\tau,\varphi,\rho$ be the three involutions in $G$. The action of $G$ stabilizes (with order 2) 72 isolated points and three K3 surfaces $\Sigma_\tau,\Sigma_\varphi,\Sigma_\rho$, each fixed by the respective involution. The three surfaces intersect pairwise in 4 points as follows: the points in $\Sigma_\tau,\Sigma_\varphi$ belong to the set of 28 isolated points fixed by $\rho$, and similarly the other pairs. The fixed locus of $G$ consists therefore of 12 points, lying in the intersection of the three K3 surfaces.
\end{proposition}
\begin{proof}
Call $\{t_1,\dots t_8\},\{q_1,\dots q_8\},\{r_1,\dots r_8\}$ the points fixed on a K3 surface respectively by $\tau,\varphi,\rho$ (these are 24 distinct points, see \cite[\S 5.8a]{Nikulin2}): the involutions $\varphi,\rho$ act on the set $\{t_1,\dots t_8\}$ exchanging them pairwise in the same way, say $t_{2i-1}\leftrightarrow t_{2i}$. Similarly, each two involutions act in the same way on the set of points fixed by the third one.\\
Again, the involution $\tau$ on $S^{[2]}$ fixes a K3 surface $\Sigma_\tau$ given by the points $[s,\tau(s)]$, and the 28 isolated points $[t_i,t_j]$ for $i\neq j\in\{1,\dots 8\}$ (similarly the other two involutions); moreover, it holds
\[\Sigma_\tau\cap\Sigma_\varphi=\{[p,\tau(p)]=[q,\varphi(q)]\}=\{[r_i,\tau(r_i)]\},\]
which are four of the isolated points fixed by $\rho$. The other pairs of fixed surfaces intersect similarly. Notice that each involution acts on each of the surfaces fixed by the other involutions (for instance, $\tau$ acts on $\Sigma_\rho$ and $\Sigma_\varphi$).
\end{proof}

\subsection{Projective examples}

\begin{remark}\label{S2notgeneral}
Let $S$ be a general member in the moduli space of projective K3 surfaces with a symplectic action of some group $G$. The Hilbert square $S^{[2]}$ is not a general member in the moduli space of K3$^{[2]}$-type manifolds with a symplectic action of $G$: indeed, $NS(S^{[2]})=NS(S)\oplus\langle -2\rangle$ is an overlattice of finite index of $\Omega_G\oplus\langle 2d\rangle\oplus \langle -2\rangle$, so it has a bigger rank than that of a general member $X$, for which $NS(X)$ is an overlattice of finite index of $\Omega_G\oplus\langle 2d\rangle$.
\end{remark}

Projective models of K3$^{[2]}$-type manifolds with a symplectic involution have been constructed in \cite{Camere} and \cite{CGKK}. We are going to adapt some of these constructions to groups of order 4.

\textbf{Actions on the Fano variety of lines of a cubic fourfold}. \\
Let $\mathcal C$ be a cubic fourfold in $\mathbb P^5$: if $\mathcal C$ is smooth, the Fano variety of lines $F(\mathcal C)$ is an hyperk\"ahler manifold of K3$^{[2]}$-type. We recall here two important results about these manifolds:
\begin{proposition}[\protect{\cite[Prop. 2, Prop. 4, Prop. 6.ii]{BD}}]
\begin{enumerate}
\item  There exists a Hodge isometry $H^4(\mathcal C, \mathbb Z) \simeq H^2(F(\mathcal C), Z)$, that maps $H^{i,j}(\mathcal C)$ to $H^{i-1,j-1}F(\mathcal C)$.
\item The Fano variety of lines of a Pfaffian cubic fourfold $\mathcal C_V$ is isomorphic to the Hilbert scheme of two points over a K3 surface $S_V$ of degree 14: this gives a Hodge isometry $H^2(F(\mathcal C_V), \mathbb Z) \simeq H^2(S^{[2]}_V, Z)\simeq H^2(S_V, Z)\oplus\langle -2\rangle$; in particular, the image of the Pl\"ucker polarization $g$ of $F(\mathcal C_V)$ is $2L-5\delta$, where $L$ is the polarization of degree 14 of $S_V$ and $\delta$ is half the class of the exceptional divisor of $S^{[2]}_V\rightarrow Sym^2(S_V)$.
\end{enumerate}
\end{proposition}

\begin{remark}\label{condizioni_Fano}
As a consequence of the first statement we get that an automorphism $\phi$ of $\mathcal C$  lifts to a symplectic automorphism of $F(\mathcal C)$ if and only if $\phi$ acts as the identity on $H^{3,1}(\mathcal C)$. As a consequence of the second, the Fano variety of lines of a general cubic fourfold $F(\mathcal C)$ is naturally polarized with a class of square 6 and divisibility 2, as these properties hold for $2L-5\delta$ in the special case of Pfaffians.
\end{remark}

In \cite{Camere} it is shown that the only case (up to a change of coordinates) in which an involution $i$ of a cubic fourfold $\mathcal C$ lifts to a symplectic involution of $F(\mathcal C)$ is when $i$ is induced by the automorphism of $\mathbb P^5$ $\phi:(x_0:x_1:x_2:x_3:x_4:x_5)\mapsto(-x_0:-x_1:x_2:x_3:x_4:x_5)$ and the cubic fourfold has equation
\[\mathcal C_2: x_0^2\lambda_1(x_2,x_3,x_4,x_5)+ x_1^2\lambda_2(x_2,x_3,x_4,x_5)+ x_0x_1\lambda_3(x_2,x_3,x_4,x_5)+\Gamma(x_2,x_3,x_4,x_5)=0\] 
where the $\lambda_i$ are linear, while $\Gamma$ is cubic. Moreover, in the same paper there is a description of the fixed locus of the symplectic involution on $F(\mathcal C_2)$: it consists of 28 points, given by the line $x_2=x_3=x_4=x_5=0$ and the 27 lines on the cubic threefold $\Gamma(x_2,x_3,x_4,x_5)=0$, and the K3 surface $\Sigma\subset\mathbb P^1\times \mathbb P^3$ described by the complete intersection
\begin{equation}\label{SigmaCamere} \Sigma:
    \begin{cases}
      \Gamma(x_2,x_3,x_4,x_5)=0\\
      x_0^2\lambda_1(x_2,x_3,x_4,x_5)+x_1^2\lambda_2(x_2,x_3,x_4,x_5)+x_0x_1\lambda_3(x_2,x_3,x_4,x_5)=0.
    \end{cases}
\end{equation}

\textbf{1: Action of $\mathbb Z/4\mathbb Z$}. 
In \cite{Fu} it is shown that (up to a change of coordinates) of all the automorphisms $\psi$ of $\mathbb P^5$ of order four such that $\psi^2=\phi$, the only one that acts on a smooth cubic hypersurface $\mathcal C_4$ such that $\psi^*|_{H^{3,1}(\mathcal C_4)}=id$ is $\psi:(x_0:x_1:x_2:x_3:x_4:x_5)\mapsto(ix_0:-ix_1:x_2:x_3:-x_4:-x_5)$, with
\[\mathcal C_4: C(x_2,x_3)+x_2Q_1(x_4,x_5)+x_3Q_2(x_4,x_5)+x_0^2\ell_1(x_4,x_5)+x_1^2\ell_2(x_4,x_5)+x_0x_1\ell_3(x_2,x_3)=0,\]
where $C$ is a cubic polynomial, $Q_i$ are quadric and $\ell_i$ are linear (it is a specialization of $\mathcal C_2$). Since this equation depends on 16 projective parameters, and the space of projectivities of $\mathbb P^5$ that commute with $\phi$ has dimension 10, $F(\mathcal C_4)$ is the general member of a family of projective K3$^{[2]}$-type manifolds with a symplectic automorphism of order 4: by Remark \ref{condizioni_Fano}, this is the one associated to the polarization $M_3(1)$ in Theorem \ref{familieshyperkahlerZ4}.\\
Following \cite[\S 7]{Camere}, we can give a model in $\mathbb P^1_{(x_0:x_1)}\times\mathbb P^3_{(x_2:\dots: x_5)}$ of the surface fixed by $\tau^2$:
\begin{equation*} \Sigma_4:
    \begin{cases}
      C(x_2,x_3)+x_2Q_1(x_4,x_5)+x_3Q_2(x_4,x_5)=0\\
      x_0^2\ell_0(x_4,x_5)+x_1^2\ell_1(x_4,x_5)+x_0x_1\ell_2(x_2,x_3)=0.
    \end{cases}
\end{equation*}
Notice that this is a specialization of $\Sigma$ \eqref{SigmaCamere}, with the additional property that it admits a symplectic involution: indeed the residual involution induced by $\psi$ fixes 8 points on $\Sigma_4$.

We conclude with the description of the fixed locus of the automorphism of order 4 $\tau$ on $F(\mathcal C_4)$ (this was partially computed in \cite{Fu}, missing one line). There are 16 lines on $\mathcal C_4$ which are fixed by $\tau\coloneqq\psi|_{\mathcal C_4}$: calling $P_0=(1:0:0:0:0:0)$, $P_1=(0:1:0:0:0:0)$, we have the six lines that join one of the points $P_0,P_1$ with one of the three solutions $\tilde P_j=(0:0:s_j:t_j:0:0)$ of the system $\{C(x_2,x_3)=0,\ x_0=x_1=x_4=x_5=0\}$, and the two lines that join $P_0$ with the solution of $\{\ell_1(x_4,x_5)=0,\ x_0=x_1=x_2=x_3=0 \}$, $P_1$ with the solution of $\{ \ell_2(x_4,x_5)=0,\ x_0=x_1=x_2=x_3=0\}$: these give the 8 fixed points of $F(\mathcal C_4)$ belonging to the K3 surface fixed by $\tau^2$. \\
Moreover, we have the lines $(0:0:0:0:x_4:x_5)$ and $(x_0:x_1:0:0:0:0)$, and the six lines that join each of the $\tilde P_j$ with each of the two solutions of $\{s_jQ_1(x_4,x_5)+t_jQ_2(x_4,x_5)=0,\ x_0=x_1=0\}$: these give the 8 points fixed by $\tau$ among the 28 isolated points fixed by $\tau^2$.

\textbf{2: Action of $(\mathbb Z/2\mathbb Z)^2$}.
Consider the action of $(\mathbb Z/2\mathbb Z)^2$ on $\mathbb P^5$ given by 
\begin{align*}
(x_0:x_1:x_2:x_3:x_4:x_5)&\xmapsto{\tau}(-x_0:-x_1:x_2:x_3:x_4:x_5)\\
&\xmapsto{\varphi}(-x_0:x_1:-x_2:x_3:x_4:x_5)
\end{align*}
and the invariant cubic fourfold for this action
\[\mathcal C_{2,2}: x_0x_1x_2+x_0^2\ell'_0(x_3,x_4,x_5)+x_1^2\ell'_1(x_3,x_4,x_5)+x_2^2\ell'_2(x_3,x_4,x_5)+C'(x_3,x_4,x_5)=0,\]
where $\ell'_i$ are linear and $C'$ is cubic (again, it is a specialization of $\mathcal C_2$); 
counting parameters, we find that $F(\mathcal C_{2,2})$ is the general member of a family of projective K3$^{[2]}$-type manifolds with a symplectic action of $(\mathbb Z/2\mathbb Z)^2$ -- the one in Theorem \ref{familieshyperkahlerKlein} associated to $M_3(1)$.\\
Again, the K3 surface fixed by $\tau$ can be described in $\mathbb P^1_{(x_0:x_1)}\times\mathbb P^3_{(x_2:\dots: x_5)}$ as:
\begin{equation*} \Sigma_{2,2}:
    \begin{cases}
      x_2^2\ell'_2(x_3,x_4,x_5)+C'(x_3,x_4,x_5)=0\\
      x_0^2\ell'_0(x_3,x_4,x_5)+x_1^2\ell'_1(x_3,x_4,x_5)+x_0x_1x_2=0.
    \end{cases}
\end{equation*}
The surfaces fixed by the other involutions in $(\mathbb Z/2\mathbb Z)^2$ are described by similar equations, obtained by permutation of the coordinates.

\begin{proposition}
There are exactly two families of K3 surfaces with a symplectic involution that admit a projective model as smooth complete intersection in $\mathbb P^1\times\mathbb P^3$, such that the involution comes from a linear action on the ambient space. The surfaces $\Sigma_4$ and $\Sigma_{2,2}$ are general members of these families.
\end{proposition}
\begin{proof}
By adjunction, a smooth complete intersection of hypersurfaces in $\mathbb P^1\times\mathbb P^3$ is a K3 surface if their bidegrees are $(2,1)$ and $(0,3)$ respectively. We can choose coordinates of the ambient space such that an involution just changes sign to a certain number of them, and we then consider the equations of the general invariant K3 surface.
Involutions that act as the identity on either component are not symplectic: this can be checked directly on the general equations. The only suitable involutions are therefore:
\begin{align*}
i_1: &(x_0:x_1)(x_2:x_3:x_4:x_5)\mapsto(x_0:-x_1)(x_2:x_3:x_4:-x_5),\\
i_2: &(x_0:x_1)(x_2:x_3:x_4:x_5)\mapsto(x_0:-x_1)(x_2:x_3:-x_4:-x_5).\\
\end{align*}
Invariant K3 surfaces for $i_1$ are described by equations that contain only monomials on which $i_1$ acts as the identity, otherwise the resulting surfaces are singular: therefore they satisfy the same equations as $\Sigma_{2,2}$. Invariant K3 surfaces for $i_2$ are given by taking equations such that one contains only monomials on which $i_2$ acts as the identity, and the other only monomials on which it acts as its opposite: indeed, taking both equations in the positive (or negative) eigenspace, we find that the resulting surface is completely fixed by $i_2$, so it has no involution acting on it. Taking equations in different eigenspaces, we obtain $\Sigma_4$ as a general member.
\end{proof}

\textbf{Non natural actions on the Hilbert square of a K3 surface}.
We recall the construction of Beauville's involution \cite{Beauville2}: it is a non-symplectic, non-natural involution defined on $S^{[2]}$ for a smooth quartic surface $S\subset\mathbb P^3$ without lines (a condition satisfied by the general smooth quartic surface).
The Hilbert square $S^{[2]}$ parametrizes unordered pairs of points $[p,q]$ of $S$; take the line $\ell\subset\mathbb P^3$ through $p$ and $q$: then $\ell\cap S=\{p,q,P,Q\}$ (not necessarily distinct). Define Beauville's involution $\beta$ on $S^{[2]}$ generically as $[p,q]\mapsto[P,Q]$. Since $S$ is general, $NS(S)$ is spanned by the hyperplane class $H$, that has self-intersection $4$; therefore $NS(S^{[2]})=\langle H\rangle\oplus\langle\mu\rangle=\langle 4\rangle\oplus\langle -2\rangle$; the isometry $\beta^*$ has invariant lattice $\langle 2\rangle$, generated by the class $H-\mu$.

\begin{proposition}\label{Beauvillecommuta}
Let $S\subset\mathbb P^3$ be a smooth quartic surface without lines, that has an automorphism $\alpha$ that preserves the polarization: then, the induced automorphism $\alpha$ on $S^{[2]}$ commutes with Beauville's involution $\beta$. \end{proposition}
\begin{proof}
The automorphism $\alpha$ is induced by an automorphism of $\mathbb P^3$, that by definition maps lines into lines, so $\beta([\alpha(p),\alpha(q)])=[\alpha(P),\alpha(Q)]$. 
\end{proof}
If $\alpha$ acts as $-id$ on $\omega_S$, then $\alpha\circ\beta$ is a non-natural symplectic automorphism of $S^{[2]}$. This construction can therefore be used to describe the general member of a projective family of K3$^{[2]}$-type manifolds with a finite symplectic action. When $\alpha$ is an involution, this has been done in \cite[Rem. 2.13]{CGKK}. To obtain a non-natural symplectic action of a group of order 4 $G$, we are going to start from K3 surfaces with a mixed action, i.e. an action of $G$ such that only a proper subgroup $K\subset G$ acts symplectically. Mixed actions of finite groups on K3 surfaces have been classified in \cite{BH}.

\textbf{1: Non natural action of $\mathbb Z/4\mathbb Z$}. Quartic surfaces with equation
\begin{align*}S_4:\ &a_1x_0^4+x_0^2(a_2x_1^2+a_3x_2x_3)+x_0x_1(a_4x_2^2+a_5x_3^2)+\\&+x_1^2(a_6x_1^2+a_7x_2x_3)+x_2^2(a_8x_2^2+a_9x_3^2)+a_{10}x_3^4=0\end{align*}
are invariant for the automorphism $\gamma:(x_0, x_1, x_2, x_3)\mapsto (x_0, -x_1, ix_2, -ix_3)$, which acts as $-id$ on its symplectic form \cite[Ex. 1.2]{AS}. The general member of this family is smooth and contains no lines (this can be checked by computer on an example, as both are open conditions). Since this family of surfaces has 6 moduli, taking the Hilbert square we obtain a general member in a projective family of projective hyperk\"ahler manifolds of type K3$^{[2]}$ with a symplectic automorphism of order 4, $\beta\circ\gamma$ (see Remark \ref{S2notgeneral}).

\begin{proposition}
Let $X$ be the general member of the projective family of K3$^{[2]}$-type manifolds associated to $\tilde M_1(1)$ in Theorem \ref{familieshyperkahlerZ4}: then $X$ is the Hilbert scheme of two points on $S_4$.
\end{proposition}
\begin{proof}
There are 5 deformation families of K3 surfaces with an automorphism of order 4 that acts as $-id$ on the symplectic form; each of them has a different fixed locus, and the corresponding invariant lattices have different ranks \cite[Prop. 2]{AS}. The automorphism $\gamma$ of $S_4$ induced from $\gamma$ on $\mathbb P^3$ described above has empty fixed locus, and its invariant lattice is $L^\gamma=U(4)\oplus D_4(2)$ (see \cite{BH} and the attached database \cite{BHdata}, entry (1.2.7.55)): therefore $NS(S_4)$ (which is recorded in the database) is an overlattice of $U(4)\oplus D_4(2)\oplus E_8(2)$, as $E_8(2)$ is the coinvariant lattice for the symplectic involution $\gamma^2$. \\
Since by Proposition \ref{Beauvillecommuta} $X=S_4^{[2]}$ admits the symplectic automorphism $\tau=\beta\circ\gamma$ of order 4, the lattice
$NS(X)=NS(S_4)\oplus\langle-2\rangle$ is isomorphic to $\Omega_4\oplus\langle2d\rangle$ or one of its overlattices; moreover, the involution $\tau^2$ on $X$ is induced by $\gamma^2$ on $S_4$, so $\Omega_2$ is embedded in $\Omega_4$ according to \cite[\S 3.2]{P}, and it holds $NS(X)^{\tau^2}=NS(S)^{\gamma^2}\oplus\langle -2\rangle=L^\gamma\oplus\langle-2\rangle$.\\ 
We find 
\[L^\gamma\oplus\langle-2\rangle\simeq\Omega_2^\perp\subset \Omega_4\oplus\langle2\rangle,\]
 so $NS(X)=\Omega_4\oplus\langle2\rangle$: notice that the $(-2)$-class $\mu$ that generates ${(L^\gamma\oplus\Omega_2)}^{\perp_{NS(X)}}$ cannot glue to $\Omega_2$, so $NS(S_4)=\mu^\perp$ in $\Omega_4\oplus\langle2\rangle$. Finally, knowing $NS(S_4)$ we can compute $T({S_4})\simeq T(X)$, and thus completely determine the deformation family of $S_4^{[2]}$: this is the one associated in Theorem \ref{familieshyperkahlerZ4} to $\tilde M_1(1)$.
\end{proof}

\textbf{2: Non natural action of $(\mathbb Z/2\mathbb Z)^2$}. The family of quartics
\[S_{2,2}: f_4(x_0,x_1)+x_2^2f_2(x_0,x_1)+x_3^2g_2(x_0,x_1)+\alpha x_2^2+\beta x_2x_3+\gamma x_3^2=0\]
admits a mixed action of $(\mathbb Z/2\mathbb Z)^2$, with generators $\sigma:(x_0, x_1, x_2, x_3)\mapsto (-x_0, -x_1, x_2, x_3)$ and $i:(x_0, x_1, x_2, x_3)\mapsto (x_0, x_1, -x_2, x_3)$: here $\sigma$ is symplectic, while $i$ and $\sigma\circ i$ are non-symplectic and each of them fixes a curve of genus 3 on $S_{2,2}$. The general member of this family is smooth and contains no line. This family has 8 moduli, so taking the Hilbert square we obtain a general member of a projective family of hyperk\"ahler manifolds of type K3$^{[2]}$ with a symplectic action of $(\mathbb Z/2\mathbb Z)^2=\langle\sigma, i\circ\beta\rangle$.
\begin{proposition}
The Hilbert square $X=S_{2,2}^{[2]}$ is a general member of a family of projective hyperk\"ahler of K3$^{[2]}$-type with a symplectic action of $(\mathbb Z/2\mathbb Z)^2$: in Theorem \ref{familieshyperkahlerKlein}, this is the family associated to the class $M_1(1)$.
\end{proposition}
\begin{proof} In the database \cite{BHdata} the collection of possible mixed actions of $(\mathbb Z/2\mathbb Z)^2$ on a K3 surface $S$ consists of 354 elements: there are two of them for which both non-symplectic involutions fix a curve of genus 3, the entries (1.2.9.13) and (1.2.9.21).\\
Since we know that one of the symplectic involutions on $X=S_{2,2}^{[2]}$ is natural, $\Omega_2$ is embedded in $\Omega_{2,2}$ as co-invariant lattice of one of the elements of $(\mathbb Z/2\mathbb Z)^2$. We can then compare the orthogonal complement to $\Omega_2$ in each of the lattices $S$ in Theorem \ref{familieshyperkahlerKlein}, with the  invariant lattices of the two possible actions: thus we can exclude (1.2.9.13), and find that $NS(X)=\Omega_{2,2}\oplus\langle2\rangle=\langle-2\rangle\oplus K$, where $K$ is the Néron-Severi lattice of the K3 surface (1.2.9.21): this is therefore our $S_{2,2}$. Since $T(X)\simeq T(S_{2,2})$, we can then find the deformation family which $X$ belongs to -- the one associated in Theorem \ref{familieshyperkahlerKlein} to $M_1(1)$.
\end{proof}

\textbf{Another example for $(\mathbb Z/2\mathbb Z)^2$}. 
Following \cite[\S 5A]{CGKK}, a projective family of K3$^{[2]}$-type manifolds with a symplectic involution $i$ can be realized as a double cover of a cone $C(\mathbb P^2\times\mathbb P^2)\subset\mathbb P^9$, with $i$ that acts exchanging the two copies of $\mathbb P^2$.\\
From a lattice-theoretic point of view, this model is given by a big and nef divisor $H=F_1+F_2\in NS(X)$ such that $\langle F_1,F_2\rangle=U(2)$ and $i^*F_1=F_2$. 
\begin{proposition}
The involution $i$ can never be the square of a symplectic automorphism of order 4 on $X$, but it can be one of the generators for a symplectic action of $(\mathbb Z/2\mathbb Z)^2$.
\end{proposition}
\begin{proof}
The same lattice-theoretic condition gives for K3 surfaces a projective model as a double cover of $\mathbb P^1\times \mathbb P^1$ with a symplectic involution $i$ that exchanges the two copies of $\mathbb P^1$: if $i$ is the square of a symplectic automorphism of order 4 it always holds $i^*F_1=F_1$ (see \cite[\S 6.3 no. 3]{P}) so this model cannot be realized; if $i$ is one of the generators of $(\mathbb Z/2\mathbb Z)^2$ however this model exists (see  \cite[\S 4.3 no. 5]{P2}).
\end{proof}

\section{The induced involutions on Nikulin orbifolds}\label{sec:Nik}

\subsection{Nikulin orbifolds and the cohomology of quotients}

Let $X$ be a K3$^{[2]}$-type manifold with a symplectic involution $i$: then by \cite{Mongardi4} its fixed locus is always a K3 surface $\Sigma$ and 28 isolated points. Let $\pi: X\rightarrow X/i$ be the quotient map: blowing up $\pi(\Sigma)$ we obtain from $X/i$ an irreducible holomorphic symplectic orbifold (IHSO) $Y$ with 28 isolated singularities, which in \cite{CGKK} is called a Nikulin orbifold. \\
Nikulin-type orbifolds (which are IHSOs deformation equivalent to Nikulin orbifolds) are one of the most studied examples of irreducible symplectic varieties. Their second integral cohomology group (i.e. that of their smooth locus) is endowed with a symplectic form, which gives it a lattice structure \cite[Thm. 2.5]{Menet2}
\begin{equation}\label{eq:cohomology_Nik_orbifold}
H^2(Y,\mathbb Z)\simeq\Lambda_{\mathrm{N}}\coloneqq U(2)^{\oplus 3}\oplus E_8\oplus\langle-2\rangle^{\oplus 2}.\end{equation}
There are results about their K\"ahler cone in \cite{MR2}\footnote{
More recently, it has been proven in \cite{BMM} the maximality of the monodromy group for Nikulin-type orbifolds: this has allowed the authors to give a complete description of the K\"ahler cone, and a lattice-theoretic classification of symplectic automorphisms (both regular and birational) for the deformation type.}, and projective families of Nikulin orbifolds have been classified in \cite{CGKK}.

More generally, one can study the terminalization $W$ of any quotient of a hyperk\"ahler manifold $X$ by a finite group $G$ acting symplectically, as they provide more examples of symplectic orbifolds \cite{BGMM}.\\
The symplectic form of $W$ (which naturally comes from the symplectic form of $X$) gives a lattice structure to $H^2(W,\mathbb Z)$. It holds:
\begin{equation}\label{eq:cohomology_of_quotient} H^2(W,\mathbb Q)\simeq (\pi_*H^2(X,\mathbb Z)\oplus E)\otimes \mathbb Q,
\end{equation}
 where $E$ is generated by the exceptional divisors introduced in the terminalization. However, it is not generally true that $\pi_*H^2(X,\mathbb Z)$ is primitively embedded in $H^2(W,\mathbb Z)$, and a proof of primitivity is fundamental in Menet's work for the computation of $\Lambda_{\mathrm N}$; moreover, there might be integral classes of the form $(x+e)/n$, with $x\in\pi_*H^2(X,\mathbb Z),\ e\in E,\ n\in\mathbb Z$ (see Remark \ref{rem:cohomology_Nik_orbifold}); the lattice $H^2(W,\mathbb Z)$ might also be odd (e.g. see \cite[Thm. 1.1]{KM}).

\begin{remark}\label{rem:cohomology_Nik_orbifold}
In the description \eqref{eq:cohomology_Nik_orbifold} of $\Lambda_{\mathrm N}$, the component $\langle-2\rangle^{\oplus 2}$ is generated by $(\tilde\mu\pm\tilde\Sigma)/2$, where $\tilde\mu$ is the image through the quotient map of the class $\mu$ generator of $\Lambda_{\mathrm{K3}}^\perp$ in $\Lambda_{\mathrm{K3}^{[2]}}\simeq H^2(X,\mathbb Z)$, and $\tilde\Sigma$ is the exceptional class introduced in the terminalization $Y\rightarrow X/i$ \cite[Thm. 2.39]{Menet2}.
\end{remark}

\subsection{Natural and standard actions}

\begin{definition}[see \protect{\cite[Def. 2.3]{Mongardi5}}]
Let $S$ be a K3 surface with a symplectic action of a finite group $G$. Then $G$ acts on $S^{[2]}$, as it acts on $Sym^2(S)$ preserving the diagonal, and therefore lifts to to its blow-up in $S^{[2]}$. This action is called \emph{natural}. \\
Let $X$ be a K3$^{[2]}$-type manifold. We call $(X, G)$ a \emph{standard pair} when $(X, G)$ can be deformed to a pair $(S^{[2]}, H)$ where $H\simeq G$ and the action of $H$ is natural.
\end{definition}

The natural action of a group $G$ on $S^{[2]}$ preserves the split 
\[H^2(S^{[2]},\mathbb Z)\simeq H^2(S,\mathbb Z)\oplus\langle -2\rangle,\]
acting on $H^2(S,\mathbb Z)$ as expected, and as the identity on its orthogonal complement. 

\begin{remark}\label{std_act}
If the action of $G$ on $X$ is standard, a similar statement holds: there exists a marking $\varphi: H^2(X,\mathbb Z)\simeq\Lambda_{\mathrm{K3}^{[2]}}=\Lambda_{\mathrm{K3}}\oplus\langle -2\rangle$ such that $G$ acts as expected on the abstract K3 lattice, and as the identity on its orthogonal complement (see also Remark \ref{G_action_standard}).
\end{remark}

Let $i_S$ be a symplectic involution on a K3 surface $S$, consider the quotient $\pi_S: S\rightarrow S/i_S$, and $\tilde S$ its terminalization. The lattice 
$\pi_{S*}H^2(S,\mathbb Z)$ is primitive in $H^2(\tilde S,\mathbb Z)$, and it is isomorphic to
\[\tilde\Lambda\coloneqq U(2)^{\oplus 3}\oplus E_8.\]
Now fix the marking $\phi:H^2(S,\mathbb Z)\simeq\Lambda_{\mathrm{K3}}$, and call $\pi_{S*}$ also the quotient map of the abstract lattice. As a symplectic involution $i$ on a K3$^{[2]}$-type manifold $X$ is always standard \cite{Mongardi4}, there exist a marking $\varphi$ as in Remark \ref{std_act}, and a unique marking $\hat\varphi$ induced by it such that the following square commutes:
\begin{equation}\label{diag_quot}
\xymatrix{
H^2(X,\mathbb Z)\ar[r]^\varphi \ar[d]_{\pi_*} & \Lambda_{\mathrm{K3}}\oplus\langle -2\rangle\ar[d]^{\pi_{S*}\oplus (\cdot 2)}\\
\pi_*H^2(X,\mathbb Z)\ar[r]_{\hat\varphi} &\tilde\Lambda\oplus\langle-4\rangle.}\end{equation}

\begin{lemma}\label{lem:cohomology_Nik_orbifold}
Let $G$ be a group containing a normal involution $i$ and acting symplectically on a K3$^{[2]}$-type manifold $X$, and assume the action is standard. Then $G/\langle i\rangle$ acts symplectically on $H^2(Y,\mathbb Z)$, and there is a marking of $Y$
\[H^2(Y,\mathbb Z)\simeq \tilde\Lambda\oplus \langle -2\rangle^{\oplus 2},\]
such that $G/\langle i\rangle$ acts as the identity on $\langle -2\rangle^{\oplus 2}$, and acts on its orthogonal complement as it does on $\pi_{S*}\Lambda_{\mathrm{K3}}$.
\end{lemma}
\begin{proof}
Consider the same diagram \eqref{diag_quot}, assuming now that $G$ acts symplectically on $S$. Then $G/\langle i\rangle$ acts on $\tilde\Lambda\oplus\langle -4\rangle$ preserving $\tilde\Lambda^\perp$. The
marking in the statement can be obtained from $\hat\varphi$ (see Remark \ref{rem:cohomology_Nik_orbifold}).
\end{proof}

If $G$ has order 4 and $X$ is of K3$^{[2]}$-type, and $Y$ is the Nikulin orbifold obtained as terminalization of $X/i$ (when $G$ is cyclic, $i$ is the square of the generator, otherwise it's any involution), then $Y$ has a residual symplectic involution $\iota$ induced by $G/\langle i\rangle$. Consider the quotient map $\pi_\iota:Y\rightarrow Y/\iota$, and $W$ the terminalization of the quotient. 
Then, we can use our knowledge of the action of $G$ on $\Lambda_{\mathrm{K3}}$ from \cite{P,P2} to give some partial results about the cohomology of $W$: by Remark \ref{G_action_standard}, we can take as $\pi_S$ in Lemma \ref{lem:cohomology_Nik_orbifold} one of the quotient maps introduced for K3 surfaces with a symplectic action of $G$: $\pi_{2*}$ if $G$ is cyclic \cite[\S 4.1]{P}, or $\pi_{\tau*},\pi_{\varphi*}$ otherwise \cite[\S 2.1, 2.2]{P2}. 
\begin{remark}
Through the push-pull map, we have some information about the primitivity of $\pi_{\iota*}x$, for $x\in H^2(Y,\mathbb Z)$: indeed, for a primitive invariant class $x=\iota(x)$, $\pi_{\iota*}x$ will be twice a primitive class in $H^2(W,\mathbb Z)$ if and only if $x=y+y'$ with $y$ primitive such that $y\neq y'=\iota(y)$. 
\end{remark}

\subsubsection{The case $G=\mathbb Z/4\mathbb Z$}

\begin{lemma}\label{quotientmapZ4hyperkahler}
Let $\tau$ be the generator of $G$, and consider the maps $\pi_2: S\rightarrow S/\tau^2$, $\pi_4: S\rightarrow S/\tau$. Recall the description of the invariant lattice $\Lambda_{\mathrm{K3}}^\tau$ and its basis from \eqref{marking_invariant_Z4}: then, it holds $\pi_{4*}\Lambda_{\mathrm{K3}}=\pi_{4*}\Lambda_{\mathrm{K3}}^\tau$. Denoting $\hat\star\coloneqq\pi_{2*}\star,\ \overline\star\coloneqq\pi_{4*}\star$, it holds:
\begin{align*}
\pi_{2*}(U\oplus\langle-2\rangle^{\oplus 2}\oplus U(4)^{\oplus 2})=& \ U(2)\oplus\langle-4\rangle^{\oplus 2}\oplus U(2)^{\oplus 2}=\\
& \ \langle\hat s_1,\hat s_2\rangle\oplus\langle\hat w_1,\hat w_2\rangle\oplus\langle\hat w_3/2,\dots,\hat w_6/2\rangle,\\[6pt]
\pi_{4*}(U\oplus\langle-2\rangle^{\oplus 2}\oplus U(4)^{\oplus 2})=& \ U(4)\oplus\langle-2\rangle^{\oplus 2}\oplus U^{\oplus 2}=\\
&\ \langle\overline s_1,\overline s_2\rangle\oplus\langle\overline w_1/2,\overline w_2/2\rangle\oplus\langle\overline w_3/4,\dots,\overline w_6/4\rangle.
\end{align*}
Moreover, it holds $\pi_{4*}=(\tilde\pi\circ\pi_{2})_*$, where $\tilde\pi:S/\tau^2\rightarrow S/\tau$.
\end{lemma}
\begin{proof}
We apply the maps $\pi_{2*}$ and $\pi_{4*}$ described in \cite[\S 4.1]{P} to the invariant lattice $\Lambda_{\mathrm {K3}}^\tau$, after a change of basis to get from its description in \eqref{marking_invariant_Z4} to the one in \cite[\S 3.2]{P}.
\end{proof}
\begin{proposition}\label{cohomology_quotient_Z4}
Let $X$ be a K3$^{[2]}$-type manifold with a symplectic action of $G=\mathbb Z/4\mathbb Z$. The terminalization $W$ of the quotient $X/G$ is a primitive symplectic orbifold, $\pi_1(W_{reg})=\mathbb Z/2\mathbb Z$, and $H^2(W, \mathbb Z)$ is an overlattice of finite index (possibly 1) of $U^{\oplus 2}\oplus U(4)\oplus\langle -2\rangle^{\oplus 2}\oplus\langle -4\rangle^{\oplus 2}$.
\end{proposition}
\begin{proof}
By Proposition \ref{fixZ4onhyperkahler} we know that $\tau$ fixes only isolated points on $X$, so the terminalization of the quotient is a primitive symplectic orbifold, of which we can compute the fundamental group of the regular locus by \cite[Prop. 7.1]{BGMM}. Let $Y$ be the Nikulin orbifold that arises as terminalization of $X/\tau^2$: since the involution $\iota$ induced by $G/\tau^2$ on the Nikulin orbifold $Y$ does not fix any surface, the terminalization $W\rightarrow Y/\iota$ does not introduce any new divisor. We can therefore use the push-pull formula to compute the intersection form of $\pi_{\iota*}H^2(Y,\mathbb Z)$, where $\pi_\iota: Y\rightarrow Y/\iota$. Up to the choice of marking, $H^2(Y,\mathbb Z)\simeq\pi_{2*}\Lambda_{\mathrm{K3}}\oplus\langle-2\rangle^{\oplus 2}$, and 
\[\pi_{\iota*}\big(\pi_{2*}\Lambda_{\mathrm{K3}}\oplus\langle-2\rangle^{\oplus 2}\big)=\pi_{4*}\big(\Lambda^\tau_{\mathrm{K3}}\big)\oplus\pi_{\iota*}\big(\langle-2\rangle^{\oplus 2}\big);\]
since the generators $(\tilde\mu\pm\tilde\Sigma)/2$ of $\langle-2\rangle^{\oplus 2}$ are invariant, the intersection form of their image in the quotient is $\langle-4\rangle^{\oplus 2}$. Therefore, by Lemma \ref{quotientmapZ4hyperkahler}, $\pi_{\iota*}H^2(Y,\mathbb Z)\simeq U(4)\oplus\langle -2\rangle^{\oplus 2}\oplus  U^{\oplus 2}\oplus\langle -4\rangle^{\oplus 2}$, which by  \eqref{eq:cohomology_of_quotient} is a sublattice of $H^2(W,\mathbb Z)$.
\end{proof} 

\subsubsection{The case $G=(\mathbb Z/2\mathbb Z)^2$}

Consider now $G=(\mathbb Z/2\mathbb Z)^2=\langle\tau,\varphi\rangle$. Having fixed a basis of $\Lambda_{\mathrm{K3}}^G$ as in \eqref{marking_invariant_Z22}, the actions of $\tau^*,\varphi^*$ have different descriptions (see  \cite[\S2.1, \S2.2]{P2}). To avoid an excess of new notation, we're going to give an explicit description only for the total quotient map $\pi_{2,2*}$ (see \cite[\S 2.4]{P2})\footnote{In loc.cit. $\pi_{2,2*}$ is defined such that the image in the cohomology of the terminalization  $\tilde S$ of $S/G$ is \emph{not} primitive. We take here directly its primitive saturation in $H^2(\tilde S,\mathbb Z)$, as we do there eventually.}.

\begin{lemma}\label{quotientmapKleinhyperkahler}
\begin{align*}
\pi_{2,2*}(U\oplus U(2)^{\oplus 2}\oplus D_4(2))=& \ \langle2\rangle\oplus\langle-2\rangle\oplus U(2)^{\oplus 2}\oplus D_4=\\
& \ \langle (\overline s_1+\overline s_2)/2,\ (\overline s_1-\overline s_2)/2\rangle\oplus\langle\overline u_1/2,\dots, \overline u_4/2\rangle\oplus\langle\nu_1,\dots \nu_4\rangle,
\end{align*}
where $\nu_1=(-\overline m_1-\overline m_2)/4,\enspace\nu_2=(\overline m_1-\overline m_2)/4,\enspace\nu_3=(\overline m_2+\overline m_4)/4,\enspace\nu_4=(\overline m_3)/2+(\overline m_1+\overline m_2)/4$.
\end{lemma}

\begin{proposition}
Let $X$ be a K3$^{[2]}$-type manifold with a symplectic action of $G=(\mathbb Z/2\mathbb Z)^2$. The terminalization $W$ of the quotient $X/G$ is an IHSO, and $H^2(W,\mathbb Z)$ is a lattice of rank 14 containing the lattice $U(2)^{\oplus 2}\oplus D_4\oplus\langle 2\rangle\oplus\langle -2\rangle\oplus\langle -4\rangle^{\oplus 2}$.
\end{proposition}
\begin{proof}
Since $G$ is generated by symplectic involutions, each fixing a surface on $X$, the terminalization of the quotient is an IHSO \cite[Prop. 7.1]{BGMM}. Let $Y$ be the Nikulin orbifold that arises as terminalization of the quotient of $X$ by any involution of $G$, suppose $\tau$. The involution $\iota$ induced on the Nikulin orbifold $Y$ fixes two surfaces, the image in $Y$ of the surfaces $\Sigma_\rho, \Sigma_\varphi$: indeed, $\tau$ acts on each of them, so they are not identified in the quotient. Therefore, the terminalization $W\rightarrow Y/\iota$ introduces two new divisors. The image of $H^2(X,\mathbb Z)$ via the quotient map $\pi_{2,2*}$ is computed similarly to Proposition \ref{cohomology_quotient_Z4}.
\end{proof}

\subsection{Induced involutions on Nikulin orbifolds}

In this section, we are going to describe the action of the induced involutions on $\Lambda_{\mathrm{N}}$ \eqref{eq:cohomology_Nik_orbifold}. At first, we take $Y$ a Nikulin orbifold arising as terminalization of $X/i$, where $X$ is a K3$^{[2]}$-type manifold and $i\in G$, a group of order 4 acting symplectically on $X$. In analogy to standard actions on K3$^{[2]}$-type manifolds, we introduce the notion of \emph{standard} involution on Nikulin-type orbifolds, and find lattice-theoretic conditions for its existence using the Torelli theorem for IHSOs \cite[Thm. 1.1]{MR1}.

\begin{proposition}\label{cor:inducedbyZ4}
The invariant and co-invariant lattices for the involution $\iota$ induced on $Y$ by $G=\mathbb Z/4\mathbb Z$ are:
\[
H^2(Y,\mathbb Z)^{\iota}= \ U(2)^{\oplus 3}\oplus \langle -4\rangle^{\oplus 2}\oplus \langle -2\rangle^{\oplus 2};\quad
{(H^2(Y,\mathbb Z)^{\iota})}^\perp= \ D_6(2).
\]
The invariant and co-invariant lattices for the involution $\iota$ induced on $Y$ by $G=(\mathbb Z/2\mathbb Z)^2$ are:
\[H^2(Y,\mathbb Z)^{\iota}= \ U^{\oplus 2}\oplus U(2)\oplus D_4(2)\oplus\langle-2\rangle^{\oplus 2};\quad
{(H^2(Y,\mathbb Z)^{\iota})}^\perp= \ D_4(2).\]
\end{proposition}
\begin{proof}
By Lemma \ref{lem:cohomology_Nik_orbifold} and Remark \ref{rem:cohomology_Nik_orbifold}, the induced involutions act on $\langle -2\rangle^{\oplus 2}$ as the identity, and on its orthogonal complement $\tilde\Lambda$ as the residual involutions described for K3 surfaces in \cite[\S 4.1]{P} for $G=\mathbb Z/4\mathbb Z$, in \cite[\S 2.3, \S 2.5]{P2} for $G=(\mathbb Z/2\mathbb Z)^2$.
\end{proof}

\begin{definition}\label{standard_involution_on_orbifold} Let $N$ be a Nikulin-type orbifold. We call an involution $\tilde\iota$ \emph{standard} if the pair $(N,\tilde\iota)$ can be deformed to a pair $(Y, \iota)$, where $Y$ is the terminalization of $X/i$ for some K3$^{[2]}$-type manifold $X$ with a symplectic action of a group $G$ of order 4, $i\in G$ is a symplectic involution, and $G/\langle i\rangle=\langle\iota\rangle$.
\end{definition}

\begin{remark}
The lattice $\Lambda_{\mathrm N}$ can be obtained as overlattice of $D_6(2)\oplus(U(2)^{\oplus 3}\oplus\langle-4\rangle^{\oplus 2}\oplus\langle-2\rangle^{\oplus 2})$ in more than one way: in other words, the embedding $D_k(2)\hookrightarrow\Lambda_{\mathrm N}$ is not determined by its orthogonal complement. The correct gluing between invariant and co-invariant lattices is equivariant for the group action: in our case, it is along a certain subgroup of the discriminant form of $D_k(2)$ isomorphic to $(\mathbb Z/2\mathbb Z)^k$, and is given in Lemma \ref{incollamenti_involuzioni_indotte} below; there are other possible gluings of these sublattices that give $\Lambda_{\mathrm N}$ as a result (involving elements of order 4 in the discriminant groups), but they do not give an involution on $\Lambda_{\mathrm N}$ .
\end{remark}

\begin{lemma}\label{incollamenti_involuzioni_indotte}
Let $N$ be a Nikulin-type orbifold which admits a standard symplectic involution $\iota$. Then one of the following holds:
\begin{enumerate}
\item the gluing between $H^2(N,\mathbb Z)^{\iota}\simeq U(2)^{\oplus 3}\oplus\langle -4\rangle^{\oplus 2}\oplus\langle-2\rangle^{\oplus 2}$ and ${(H^2(N,\mathbb Z)^{\iota})}^\perp\simeq D_6(2)$ which gives 
$H^2(N,\mathbb Z)$ as overlattice is obtained by adding as generators the following elements: $(d_5+d_6+u_1+u_2+v_2)/2,\ (d_5+a_1+u_1+v_2)/2,\ (d_1+d_4+d_6+a_1)/2,\ (d_3+d_6+v_1)/2,\ (d_2+d_4+d_6+a_2)/2,\ (d_1+d_3+a_2+v_2)/2$; here $\{d_1,\dots,d_6\}$ are the generators of $D_6(2)$ numbered so that $d_id_3=2$ for $i=1,2,4$, $\{u_1,u_2\}$ and $\{v_1,v_2\}$ are the generators of two of the copies of $U(2)$, $\{a_1,a_2\}$ of $\langle -4\rangle^{\oplus 2}$; this corresponds to the case where $\iota$ is induced by $\mathbb Z/4\mathbb Z$;
\item the gluing between $H^2(N,\mathbb Z)^{\iota}\simeq U^{\oplus 2}\oplus U(2)\oplus D_4(2)\oplus\langle-2\rangle^{\oplus 2}$ and ${(H^2(N,\mathbb Z)^{\iota})}^\perp\simeq D_4(2)$ which gives 
$H^2(N,\mathbb Z)$ as overlattice is obtained by adding as generators the following elements: $
(d_1+d_2+e_2+e_4)/2,\ (d_2+d_3+e_4)/2,\ (d_2+e_1+e_3+e_4)/2,\ (d_4+e_2+e_3+e_4)/2$;  here $\{d_1,\dots,d_4\}$ and $\{e_1,\dots,e_4\}$ are the generators of the two copies of $D_4(2)$, and $d_3,e_3$ are the center of the respective Dynkin diagram; this corresponds to the case where $\iota$ is induced by $(\mathbb Z/2\mathbb Z)^2$.
\end{enumerate}
\end{lemma}

\begin{remark}\label{stellina}
We say that an embedding $\varphi: D_6(2)\hookrightarrow \Lambda_\mathrm{N}$  satisfies the condition ($\star$)
if the gluing between $\varphi(D_6(2))$  and its orthogonal complement is as in Lemma \ref{incollamenti_involuzioni_indotte}.1. \\
Similarly, we say that $\psi: D_4(2)\hookrightarrow \Lambda_\mathrm{N}$ satisfies the condition ($\star$) if the gluing between $\psi(D_4(2))$ and its orthogonal complement is as in Lemma \ref{incollamenti_involuzioni_indotte}.2.
\end{remark}

Let now $\mathcal M$ be the moduli space of marked Nikulin-type orbifolds (see \cite[\S 2.2]{MR2}), and let $\mathcal M_0$ be the connected component of a Nikulin orbifold $Y$ with an induced symplectic involution.

\begin{theorem}\label{existence_of_standard_involution_on_Y} 
Let $N$ be a Nikulin-type orbifold such that there is an embedding of $D_k(2)$ in $NS(N)$, with $k$ either 4 or 6, such that the induced embedding $\varphi: D_k(2)\hookrightarrow H^2(N, \mathbb Z)$ satisfies the condition ($\star$), and assume $N\in \mathcal M_0$: then $N$ admits a standard symplectic involution.
\end{theorem}

\begin{proof}
We define the isometry $\alpha$ on $H^2(N,\mathbb Z)$ that acts as $-id$ on $\varphi(D_k(2))$, and as the identity on its orthogonal complement.
To apply the Torelli theorem for IHSOs \cite[Thm. 1.1]{MR1} we need to show that $\alpha$ is 1) an integral Hodge isometry which 2) is a monodromy operator and 3) preserves the K\"ahler cone. The first condition is satisfied by construction, because $T(N)$ is contained in the orthogonal complement to $\varphi(D_k(2))$, so $\alpha$ acts on it as the identity; the second is true by the assumption $N\in \mathcal M_0$. To prove 3), refer to the description of the walls of the K\"ahler cone of $N$ in \cite{MR2}. The lattice $\varphi(D_6(2))$ does not contain any wall-divisor: indeed it contains no $(-2)$-classes and, if we assume $E_8\oplus U(2)^{\oplus 3}=\langle e_1,\dots,e_8\rangle\oplus\langle u_{i,1},u_{i,2}\rangle_{i=1,2,3}$, we get as generators of $\varphi(D_6(2))$
\begin{align*}
d_1&=2e_1+2e_2+5e_3+7e_4+5e_5+4e_6+3e_7+e_8+u_{2,1}+2u_{2,2}+u_{3,1}-u_{3,2}, \\
d_2&=-12e_1-8e_2-17e_3-25e_4-21e_5-16e_6-11e_7-5e_8-u_{3,1}-5u_{3,2}, \\
d_3&=-6e_1-4e_2-8e_3-12e_4-9e_5-8e_6-6e_7-3e_8-u_{2,2}-u_{3,1}-u_{3,2}, \\
d_4&=3e_1+2e_2+3e_3+5e_4+4e_5+4e_6+3e_7+1e_8-u_{2,1}+2u_{3,2}, \\
d_5&=16e_1+11e_2+23e_3+33e_4+27e_5+22e_6+15e_7+8e_8+u_{2,1}+2u_{3,1}+5u_{3,2},\\
d_6&=e_1+e_5-u_{2,1},
\end{align*}
so no elements of square $-4,-6,-12$ in $\varphi (D_6(2))$ have divisibility 2; a similar result holds for $\varphi(D_4(2))$).\\
If $NS(N)=D_k(2)\oplus\langle 2d\rangle$ ($k=4,6, d\in\mathbb Z_{>0}$) or one of its overlattices, then wall classes in its Néron-Severi will have the form $aL+bv$, with $a,b\in\mathbb Q\setminus\{0\}$, $v\in D_k(2)$ and $L$ the generator of $D_k(2)^{\perp_{NS(N)}}$: since $\alpha$ does not reflect any of these classes, the K\"ahler cone of the general projective $N$ is preserved.
\end{proof}
\begin{remark}
To extend this result to the whole deformation class of Nikulin-type orbifolds, the knowledge of the monodromy group is needed. In \cite[Thm. 3.5]{BMM} it its proven that indeed Mon$^2(\Lambda_{\mathrm N}) = O^+ (\Lambda_{\mathrm N})$, so the result holds in full generality.
\end{remark}

\section{Classifying Nikulin-type Orbifolds with a standard symplectic involution}

In the first part of this section, we classify projective Nikulin orbifolds with an induced symplectic involution: these are obtained as quotients of projective K3$^{[2]}$-type manifolds with a symplectic action of a group of order four $G$, whose projective families are classified in Sections \ref{K3[2]withZ4}, \ref{K3[2]withKlein}.

We then classify projective Nikulin-type orbifolds $N$ with an action of $(\mathbb Z/2\mathbb Z)^2=\langle\iota,\kappa\rangle$, where $\iota$ is standard (see Definition \ref{standard_involution_on_orbifold}), and $\kappa$ is the non-standard involution defined in \cite{MR2} (see Theorem \ref{MR_involution}): therefore, $NS(N)$ admits a primitive embedding of either $D_6(2)\oplus\langle -2\rangle$, or $D_4(2)\oplus\langle -2\rangle$, with the additional property that the embedding $D_k(2)\hookrightarrow H^2(N,\mathbb Z)$ satisfies condition $(\star)$ (see Remark \ref{stellina}).

\subsection{Families of Nikulin orbifolds with an induced symplectic involution}

The general (non-projective) Nikulin orbifold with an induced involution has Néron-Severi lattice $\Omega_\iota\oplus\langle -4\rangle$, where $\Omega_\iota=D_6(2)$ if $\iota$ is induced by $G=\mathbb Z/4\mathbb Z$, $\Omega_\iota=D_4(2)$ if $\iota$ is induced by $G=(\mathbb Z/2\mathbb Z)^2$; the class that generates $\Omega_\iota^\perp$ is the exceptional class $\tilde\Sigma$ that is introduced with the terminalization $Y\rightarrow X/i$.

\begin{remark}\label{rem:split}
Recall that, since the symplectic action of groups $G$ of order 4 on a K3$^{[2]}$-type manifold $X$ is standard, it preserves the split
$H^2(X,\mathbb Z)\simeq \Lambda_{\mathrm{K3}}\oplus\langle -2\rangle,$
acting as the identity on $\mu$ (the generator of $\Lambda_{\mathrm{K3}}^\perp$), and as it acts on the cohomology of a K3 surface $S$ on $\Lambda_{\mathrm{K3}}$. Therefore, the map induced by $\pi: X\rightarrow X/i$ can be recovered from $\pi_S: S\rightarrow S/i$.
\end{remark}
{\sffamily The case $G=\mathbb Z/4\mathbb Z$}

In the following table we classify the projective families of Nikulin orbifolds $Y$ admitting a natural involution induced by an action of $\mathbb Z/4\mathbb Z$ on a K3$^{[2]}$-type manifold $X$, by giving the possible pairs $(NS(Y),T(Y))$ in the second and third column. Starting from the projective families of $X$ classified in Theorem \ref{familieshyperkahlerZ4}, we apply the map induced in cohomology by $\pi: X\rightarrow X/i$ as in Remark \ref{rem:split}, using
$\pi_{S*}=\pi_{2*}$ as described in \cite[\S 4.1]{P}.
If $X$ is polarized with a class $L$ of square $2d$, then on $Y$ we consider $H=\pi_*L$, or $H=\pi_*L/2$ if the former is not primitive: therefore, it holds $H^2=4d$ or $H^2=d$ respectively. The relation between the number $d$ in the first column and $m$ appearing in the last columns is given in the proof of Theorem \ref{familieshyperkahlerZ4}, where the classes $M_i(m)$ are constructed; since in many cases the same lattice $T(Y)$ obtained using $H=\pi_*M_i(m)$, therefore depending on $m$, is obtained also using $L_i(h)$, we write $L_i(m)$ accordingly (see \cite[Ex. 5.1.6]{P}). The relation between $m, j, h$ appearing in the last line of the table is explained in Table \ref{mjh2}.

\begin{remark}\label{SigmaM}
Notice that for the families with a polarization $L=L_S\in\mu^\perp$, 
i.e. those families such that $T(X)=T(S)\oplus\langle -2\rangle$ for some general projective K3 surface $S$ admitting a symplectic action of $G$, the class $\tilde\Sigma$ does not glue to any element in $NS(Y)$; for the families with a polarization $M=\lambda+n\mu$ instead, $\tilde\Sigma$ glues to $\pi_*M$.
\end{remark}
\newpage
\footnotesize
\begin{longtable}{|c|c|c|c|c|}\caption{Projective families of Nikulin orbifolds in the case $G=\mathbb Z/4\mathbb Z$}\label{NikOrbZ4}\\
\cline{2-5}
\nocell{1} &{$NS(Y)$\Tstrut} &{$T(Y)$ \Tstrut} & $H$ &$H^2$\\ [6pt]
\hline
\multirow{2}{*}{$d=_4 1$\Tstrut}	&{$D_6(2)\oplus\langle4d\rangle\oplus\langle-4\rangle$\Tstrut} & $\langle-4d\rangle\oplus\langle-4\rangle^{\oplus 3}\oplus U(2)^{\oplus 2}$\Tstrut &$\pi_*L_0(d)$ &$4d$\\
\ &$(D_6(2)\oplus\langle4d\rangle\oplus\langle-4\rangle)'$\Tstrut & $C_m\oplus U(2)^{\oplus 2}$ &$\pi_*\ptwiddle{M}_1(m)$ &$4(4m-3)$ \\[6pt]
\hline
\multirow{4}{*}{$d=_4 2$\Tstrut} & {$D_6(2)\oplus\langle4d\rangle\oplus\langle-4\rangle$\Tstrut} &$\langle-4d\rangle\oplus\langle-4\rangle^{\oplus 3}\oplus U(2)^{\oplus 2}$\Tstrut &$\pi_*L_0(d)$ &$4d$\\[6pt]
\cdashline{2-5}
\ &$(D_6(2)\oplus\langle4d\rangle\oplus\langle-4\rangle)'$\Tstrut &\multirow{2}{*}{$B_m\oplus\langle-4\rangle\oplus U(2)^{\oplus 2}$\Tstrut} &$\pi_*M_2(m)$ &$4(4m-2)$\\
\ &$(D_6(2)\oplus\langle4d\rangle)'\oplus\langle-4\rangle$\Tstrut &\ &$\pi_*L_{2,2}^{(1,2)}(m)$ &$4(4m\pm2)$ 
\\ [6pt]
\hline
\multirow{4}{*}{$d=_4 3$\Tstrut}&{$D_6(2)\oplus\langle4d\rangle\oplus\langle-4\rangle$\Tstrut} & $\langle-4d\rangle\oplus\langle-4\rangle^{\oplus 3}\oplus U(2)^{\oplus 2}$\Tstrut &$\pi_*L_0(d)$ &$4d$\\[6pt]
\cdashline{2-5}
\ &$D_6(2)\oplus(\langle4d\rangle\oplus\langle-4\rangle)'$\Tstrut &\multirow{2}{*}{$F_m\oplus\langle-4\rangle^{\oplus 2}\oplus U(2)^{\oplus 2}$\Tstrut} &$\pi_*M_3(m)$ &$4(4m-1)$\\
\ &$(D_6(2)\oplus\langle4d\rangle)'\oplus\langle-4\rangle$\Tstrut &\ &$\pi_*L_{2,3}(m)$ &$4(4m+3)$\\[6pt]
\hline
\multirow{5}{*}{$d=_4 0$\Tstrut} & $D_6(2)\oplus\langle4d\rangle\oplus\langle-4\rangle$\Tstrut &$\langle-4d\rangle\oplus\langle-4\rangle^{\oplus 3}\oplus U(2)^{\oplus 2}$\Tstrut  &$\pi_*L_0(d)$ &$4d$   \\[6pt]
\cdashline{2-5}
\ &$D_6(2)\oplus\langle d\rangle\oplus\langle-4\rangle$\Tstrut &\multirow{4}{*}{$G_m\oplus\langle-4\rangle^{\oplus 2}\oplus U(2)^{\oplus 2}$\Tstrut}  &$\pi_*L_{2,0}(m)/2$ &$4(m-1)$\\
\ &$(D_6(2)\oplus\langle d\rangle\oplus\langle-4\rangle)'$\Tstrut &\ &$\pi_*M_4(m)/2$ &$4(m-1)$\\
\ & $(D_6(2)\oplus\langle d\rangle)'\oplus\langle-4\rangle$\Tstrut  &\  &$\pi_*L_{4,j}(h)/2$ &$4(m-1)$, see Table \ref{mjh2}\\ [6pt]
\hline
\end{longtable}
\begin{table}[h!]
\caption{Relation between 
$m,j,h$}\label{mjh2}
\centering
\begin{tabular}{c|c c c c}
$m$ (mod 4) & 0 & 1 & 2 & 3 \\
\hline
$j$ & 12 & 0 & 4 & 8\\
$h$ & $(m-4)/4$ &$(m+3)/4$  &$(m-2)/4$ &$(m+13)/4$
\end{tabular}\end{table}
\scriptsize
\begin{gather*}
B_m=\left[\begin{array}{r r r}-4m &2 &2\\ 2 &-4 &0\\2 &0 &-4\end{array}\right]\quad C_m=\left[\begin{array}{r r r r}-4m &2 &2 &2\\ 2 &-4 &0 &0\\2 &0 &-4 &0\\ 2 &0 &0 &-4\end{array}\right]\quad F_m=\left[\begin{array}{r r}-4m &2 \\ 2 &-4 \end{array}\right]\quad G_m=\left[\begin{array}{r r}-4m &4 \\ 4 &-4 \end{array}\right]\end{gather*}
\normalsize

\begin{remark}
The projective families of K3$^{[2]}$-type manifolds associated to the polarizations $M_1$ and $\tilde M_1$ give Nikulin orbifolds that belong to the same projective family. A similar statement holds for $L_{2,2}^{(1)}$ and $L_{2,2}^{(2)}$. These are the only families for which this phenomenon happens.
\end{remark}

{\sffamily The case $G=(\mathbb Z/2\mathbb Z)^2$}

In the following table we classify the projective families of Nikulin orbifolds $Y$ admitting a natural involution induced by an action of $(\mathbb Z/2\mathbb Z)^2=\langle\tau,\varphi\rangle$ on a K3$^{[2]}$-type manifold $X$, by giving the possible pairs $(NS(Y),T(Y))$ in the second and third column.

Starting from the projective families of $X$ classified in Theorem \ref{familieshyperkahlerKlein}, we apply the map induced in cohomology by $\pi: X\rightarrow X/i$ as in Remark \ref{rem:split}, with $\pi_{S*}$ equals to either $\pi_{\tau*}$ or $\pi_{\varphi*}$ \cite[\S2.1, \S2.2]{P2}: these two maps may act differently on the polarization of $X$, and therefore the same projective family of K3$^{[2]}$-type manifolds can give rise to more than one family of Nikulin orbifolds (it happens also for projective K3 surfaces, see \cite[Thm. 3.3.4]{P2}). When there is no difference between $\pi_{\tau*}$ and $\pi_{\varphi*}$, we use the notation $\pi_{\iota*}$.
On $Y$ we consider $H=\pi_*L$, or $H=\pi_*L/2$ if the former is not primitive: therefore, it holds $H^2=4d$ or $H^2=d$ respectively. The same considerations as in Remark \ref{SigmaM} can be applied here.

\begin{remark} 
The classes $\pi_{\iota*}L_0(e)$ and $\pi_{\iota*}L_{2,0}^{(2)}(e+1)/2$ give the same projective family of Nikulin orbifolds; in other words, the projective families of K3$^{[2]}$-type manifolds associated to the polarizations $L_0(e)$ and $L_{2,0}^{(2)}(e+1)$ give Nikulin orbifolds that belong to the same projective family, independently of the involution in $(\mathbb Z/2\mathbb Z)^2$ we choose when taking the quotient. These are the only families for which this phenomenon happens.
\end{remark}

\footnotesize
\begin{longtable}{|c|c|c|c|c|}\caption{Projective families of Nikulin orbifolds in the case $G=(\mathbb Z/2\mathbb Z)^2$}\label{NikOrbKlein}\\
\cline{2-5}
\nocell{1} &{$NS(Y)$\Tstrut} &{$T(Y)$ \Tstrut} & $H$ &$H^2$\\ [6pt]
\hline
{$d=_2 1$\Tstrut}	&{$D_4(2)\oplus\langle 2d\rangle\oplus\langle-4\rangle$\Tstrut} & $U\oplus\langle-4\rangle\oplus D_4(2)\oplus Q_h$\Tstrut &$\pi_{\tau*}L_{2,2}^{(b)}(h)/2$ &$2(2h+1)$\\[6pt]
\hline
\multirow{5}{*}{$d=_{4} 0$\Tstrut}	&\multirow{2}{*}{$D_4(2)\oplus\langle 4d\rangle\oplus\langle-4\rangle$\Tstrut} 
& $U\oplus\langle-4\rangle\oplus T_h$\Tstrut &$\pi_{\varphi*}L_{2,0}^{(1)}(h)/2$ &$4h$ \\
\ &\ & $\langle-4d\rangle\oplus U^2\oplus D_4(2)\oplus\langle-4\rangle$\Tstrut &$\pi_{\iota*}L_0(d)$ &$4d$\\[6pt]
\cdashline{2-5}
\ &{$(D_4(2)\oplus\langle 4d\rangle\oplus\langle-4\rangle)'$\Tstrut} & $\langle-4d\rangle\oplus U^2\oplus D_4(2)\oplus\langle-4\rangle$\Tstrut &$\pi_{\iota*}M_8(m)/2$ &$8(m-1)$\\
\ &{$(D_4(2)\oplus\langle 4d\rangle)'\oplus\langle-4\rangle$\Tstrut} & $U\oplus\langle-4\rangle\oplus\langle2\rangle\oplus P_h$\Tstrut &$\pi_{\tau*}L_{2,0}^{(1)}(h)$ &$16h$\\[6pt]
\hline
\multirow{6}{*}{$d=_{4} 1$\Tstrut}	&\multirow{2}{*}{$D_4(2)\oplus\langle 4d\rangle\oplus\langle-4\rangle$\Tstrut} 
& $U\oplus\langle-4\rangle\oplus T_h$\Tstrut &$\pi_{\varphi*}L_{2,0}^{(1)}(h)/2$ &$4h$ \\
\ &\ & $\langle-4d\rangle\oplus U^2\oplus D_4(2)\oplus\langle-4\rangle$\Tstrut &$\pi_{\iota*}L_0(d)$ &$4d$\\[6pt]
\cdashline{2-5}
\ &$(D_4(2)\oplus\langle 4d\rangle)'\oplus\langle-4\rangle$\Tstrut& $\langle-4d\rangle\oplus U^2\oplus D_4(2)\oplus\langle-4\rangle$\Tstrut &$\pi_{\iota*}L_{4,4}(h)/2$ &$16h+4$\\
\ &$(D_4(2)\oplus\langle 4d\rangle\oplus\langle-4\rangle)'$\Tstrut& $U^2\oplus R_m$\Tstrut &$\pi_{\tau*}M_1(m)$ &$4(4m-3)$\\
\ &$D_4(2)\oplus(\langle 4d\rangle\oplus\langle-4\rangle)'$\Tstrut& $B\oplus F_m$\Tstrut &$\pi_{\varphi*}M_1(m)$ &$4(4m-3)$\\[6pt]
\hline
\multirow{7}{*}{$d=_{4} 2$\Tstrut} &\multirow{2}{*}{$D_4(2)\oplus\langle 4d\rangle\oplus\langle-4\rangle$\Tstrut} 
& $U\oplus\langle-4\rangle\oplus T_h$\Tstrut &$\pi_{\varphi*}L_{2,0}^{(1)}(h)/2$ &$4h$ \\
\ &\ & $\langle-4d\rangle\oplus U^2\oplus D_4(2)\oplus\langle-4\rangle$\Tstrut &$\pi_{\iota*}L_0(d)$ &$4d$\\[6pt]
\cdashline{2-5}
\ &{$(D_4(2)\oplus\langle 4d\rangle\oplus\langle-4\rangle)'$\Tstrut} &$\langle-4d\rangle\oplus U^2\oplus D_4(2)\oplus\langle-4\rangle$\Tstrut &$\pi_{\iota*}M_8(m)/2$ &$8(m-1)$\\[6pt]
\cdashline{2-5}
\ &\multirow{2}{*}{$(D_4(2)\oplus\langle 4d\rangle)'\oplus\langle-4\rangle$\Tstrut} &{$U\oplus\langle-4\rangle\oplus W_h$\Tstrut} &$\pi_{\iota*}L_{2,2}^{(a)}(h)$ &$16h+8$ \\ 
\ &\ &$D_4\oplus\langle-4\rangle\oplus V_h$\Tstrut   &$\pi_{\varphi*}L_{2,2}^{(b)}(h)$ &$16h+8$\\ [6pt]
\hline
\multirow{6}{*}{$d=_{4} 3$\Tstrut} &\multirow{2}{*}{$D_4(2)\oplus\langle 4d\rangle\oplus\langle-4\rangle$\Tstrut} 
& $U\oplus\langle-4\rangle\oplus T_h$\Tstrut &$\pi_{\varphi*}L_{2,0}^{(1)}(h)/2$ &$4h$ \\
\ &\ & $\langle-4d\rangle\oplus U^2\oplus D_4(2)\oplus\langle-4\rangle$\Tstrut &$\pi_{\iota*}L_0(d)$ &$4d$\\[6pt]
\cdashline{2-5}
\ &{$(D_4(2)\oplus\langle 4d\rangle)'\oplus\langle-4\rangle$\Tstrut} &$\langle-4d\rangle\oplus U^2\oplus D_4(2)\oplus\langle-4\rangle$ &$\pi_{\iota*}L_{4,- 4}(h)/2$ &$16h- 4$\Tstrut\\ [6pt]
\cdashline{2-5}
\ &{$D_4(2)\oplus(\langle 4d\rangle\oplus\langle-4\rangle)'$\Tstrut} &$U^2\oplus D_4(2)\oplus S_m$\TBstrut &$\pi_{\iota*}M_3(m)$ &$4(4m-1)$\\[6pt]
\hline
\end{longtable}
\begin{gather*}P_h=\left[\begin{array}{r r r r r r}-4-4h &2 &2 &0 &0 &0\\
     2 &-2 &0 &0 &0 &0\\
     2 &0 &-4 &0 &2 &-2\\
     0 &0 &0 &-4 &2 &2\\
     0 &0 &2 &2 &-4 &0\\
     0 &0 &-2 &2 &0 &-4\end{array}\right]\quad Q_h=\left[\begin{array}{r r r} -4h &2 &0\\ 2 &0 &2\\ 0 &2 &-2\end{array}\right]\quad S_m=\left[\begin{array}{r r} -4m &2 \\2 &-4\end{array}\right]
\\ R_m=\left[\begin{array}{r r r r r r}-4m &2 &2 &2 &-2 &0\\
     2 &-4 &0 &0 &2 &0\\
     2 &0 &-4 &0 &2 &0\\
     2 &0 &0 &-4 &0 &0\\
     -2 &2 &2 &0 &-4 &2\\
     0 &0 &0 &0 &2 &-4\end{array}\right]\quad T_h=\left[\begin{array}{r r r r r r r}-4-4h &2 &0 &0 &0 &0 &0\\
     2 &0 &0 &2 &2 &0 &0\\
     0 &0 &-8 &0 &4 &0 &0\\
     0 &2 &0 &0 &4 &-2 &0\\
     0 &2 &4 &4 &0 &2 &0\\
     0 &0 &0 &-2 &2 &-4 &2\\
     0 &0 &0 &0 &0 &2 &-4\end{array}\right]\\
V_h=\left[\begin{array}{r r r r r}-4h &2 &0 &0 &0\\
     2 &8 &4 &0 &0\\
     0 &4 &2 &1 &0\\
     0 &0 &1 &-2 &1\\
     0 &0 &0 &1 &-2
\end{array}\right] \quad W_h=\left[\begin{array}{r r r r r r r}-4-4h &2 &0 &0 &0 &0 &0\\
     2 &0 &-2 &0 &0 &0 &0\\
     0 &-2 &0 &0 &2 &0 &0\\
     0 &0 &0 &-4 &0 &2 &2\\
     0 &0 &2 &0 &-4 &2 &-2\\
     0 &0 &0 &2 &2 &-4 &0\\
     0 &0 &0 &2 &-2 &0 &-4\end{array}\right]\\
F_m=\left[\begin{array}{r r r r r r}-4m &2 &0 &0 &0 &0\\
     2 &-8 &4 &0 &-2 &2\\
     0 &4 &-4 &2 &4 &0\\
     0 &0 &2 &-4 &-4 &0\\
     0 &-2 &4 &-4 &-6 &0\\
     0 &2 &0 &0 &0 &-6\end{array}\right]
\end{gather*}
\normalsize

\subsection{Nikulin-type orbifolds with a symplectic action of $(\mathbb Z/2\mathbb Z)^2$}

We now describe Nikulin-type orbifolds with a particular symplectic action of $(\mathbb Z/2\mathbb Z)^2$, generated by a standard involution, and the non-standard involution described in \cite{MR2} which consists in the reflection on a class of square $-2$ and divisibility 2.

\begin{theorem}[\protect{\cite[Thm. 8.5]{MR2}}]\label{MR_involution}
Let $Y$ be a Nikulin-type orbifold such that there exists $D\in NS(Y)$ with $D^2 = -2$ and $div(D) = 2$. Then there exists an irreducible symplectic orbifold $Z$
bimeromophic to $Y$ and a non-standard symplectic involution $\kappa$ on $Z$ such that:
\[H^2(Z,\mathbb Z)^\kappa\simeq U(2)^{\oplus 3}\oplus E_8\oplus\langle -2\rangle, \quad \Omega_\kappa\coloneqq {(H^2(Z, \mathbb Z)^{\kappa})}^\perp\simeq\langle -2\rangle.\]
\end{theorem}

\begin{remark}
If $Y$ is a Nikulin orbifold obtained as terminalization of a natural pair $(S^{[2]},i)$, assuming that $S$ is very general (i.e. $NS(S)=\Omega_i=E_8(2)$), then $NS(Y)=\langle-2\rangle^2$: therefore $\kappa$ exists on $Y$, and it acts exchanging the exceptional classes $\tilde\Sigma$ and $\pi_*\delta$ \cite[Prop. 4.5]{MR2}.
\end{remark}

\begin{corollary}
Let $N$ be a Nikulin-type orbifold that admits a symplectic action of $(\mathbb Z/2\mathbb Z)^2=\langle\iota,\kappa\rangle$, where $\iota$ is standard and $\kappa$ is the non-standard involution described in Theorem \ref{MR_involution}: then the following conditions hold: 
\begin{enumerate}
\item there exists a primitive embedding of $D_6(2)\oplus\langle-2\rangle$ or $D_4(2)\oplus\langle-2\rangle$ in $NS(N)$;
\item the resulting embedding $D_k(2)\hookrightarrow NS(N)\hookrightarrow H^2(N,\mathbb Z)$ satisfies condition ($\star$) (see Remark \ref{stellina}).
\end{enumerate}
\end{corollary}
\begin{proof}
The co-invariant lattice $\Omega_\kappa$ of the non-standard involution $\kappa$ can always be embedded in the invariant lattice of the standard involutions in a way compatible with Lemma \ref{incollamenti_involuzioni_indotte}: use one of the orthogonal $\langle-2\rangle$ components. With this choice, $\kappa$ commutes with the standard involution on $N$: therefore, we get a symplectic action of $(\mathbb Z/2\mathbb Z)^2$ on $N$. Moreover, this is the only valid choice: indeed, $\kappa$ exists if and only its co-invariant lattice is embedded with divisibility 2 through $\langle-2\rangle\hookrightarrow NS(N)\hookrightarrow H^2(N,\mathbb Z)$, and all embeddings $\varphi: \langle-2\rangle\hookrightarrow H^2(N,\mathbb Z)$ with divisibility 2 are equivalent; if $\iota$ is either standard involution, $\Omega_\iota$ is embedded in $\varphi(\langle-2\rangle)^\perp$ such that Lemma \ref{incollamenti_involuzioni_indotte} is satisfied.
\end{proof}

\begin{theorem}\label{mixedNikZ4}
Let $N$ be a general projective Nikulin-type orbifold with a symplectic action of $\mathfrak K=\langle\iota,\kappa\rangle$, where $\kappa$ is the non-standard involution described in Theorem \ref{MR_involution}
and $(N,\iota)$ is a deformation of the natural pair $(Y,\tilde\iota)$, where $ Y$ is the terminalization of $X/i$, $X$ is a K3$^{[2]}$-type manifold with a symplectic action of $G=\mathbb Z/4\mathbb Z$ and $i\in G$ of order 2. Then $N$ belongs to one of the projective families described in the following table. 
\end{theorem}
\footnotesize
\begin{longtable}{|c|c|c|c|}
\hline
{$NS(N)$\Tstrut} & $e$ &{$T(N)$ \Tstrut} &$k$ \\ [6pt]
\hline
\multirow{16}{*}{$D_6(2)\oplus\langle2e\rangle\oplus\langle-2\rangle$\Tstrut} &$e=_2 0$ & $\langle-2k\rangle\oplus U(2)^{\oplus 2}\oplus\langle-4\rangle^{\oplus 2}\oplus\langle-2\rangle$\TBstrut &$e$  \\
\cline{2-4}
\  &\multirow{2}{*}{$e=_8 1$}& $G_k\oplus U(2)^{\oplus 2}\oplus\langle-4\rangle^{\oplus 2} $ &$(e+1)/2$\Tstrut  \\[6pt]
\ &\ & $R'_k\oplus\langle4\rangle$ &$(e-1)/8$  \\[6pt]
\cline{2-4}
\  &\multirow{4}{*}{$e=_8 3$}& $G_k\oplus U(2)^{\oplus 2}\oplus\langle-4\rangle^{\oplus 2} $ &$(e+1)/2$\Tstrut \\[6pt]
\ &\ & $M'_k\oplus U(2) $ &$(e+5)/8$   \\[6pt]
\ &\ & $Q'_k\oplus\langle-4\rangle\oplus\langle4\rangle $ &$ (e+29)/8$  \\[6pt]
\cline{2-4}
\  &\multirow{4}{*}{$e=_8 5$}& $G_k\oplus U(2)^{\oplus 2}\oplus\langle-4\rangle^{\oplus 2} $ &$(e+1)/2$\Tstrut  \\[6pt]
\ &\ & $N'_k\oplus\langle-4\rangle $ &$(e-5)/8$   \\[6pt]
\ &\ & $P'_k\oplus U(2)\oplus\langle-4\rangle $ &$(e-5)/8$   \\[6pt]
\cline{2-4}
\  &\multirow{2}{*}{$e=_8 7$}& $G_k\oplus U(2)^{\oplus 2}\oplus\langle-4\rangle^{\oplus 2} $ &$(e+1)/2$\Tstrut  \\[6pt]
\ &\ & $S'_k\oplus U(2)\oplus\langle-4\rangle $ &$ (e+9)/8$   \\[6pt]
\hline
\multirow{6}{*}{$(D_6(2)\oplus\langle4e\rangle)'^{(1)}\oplus\langle-2\rangle$\Tstrut} &$e=_4 0$ &$S_k\oplus U(2)\oplus\langle-4\rangle\oplus\langle-2\rangle $ &$e/4+1$\Tstrut  \\
\ &$e=_4 1$ &$R_k\oplus\langle4\rangle\oplus\langle-2\rangle $ &$ (e-1)/4$\Tstrut  \\
\ &$e=_4 2$ &$Q_k\oplus\langle-4\rangle\oplus\langle4\rangle\oplus\langle-2\rangle $ &$(e+14)/4$\Tstrut  \\
\ &$e=_4 3$ &$P_k\oplus U(2)\oplus\langle-4\rangle\oplus\langle-2\rangle $ &$ (e-3)/4$\Tstrut \\[6pt]
\hline
\multirow{4}{*}{$(D_6(2)\oplus\langle4e\rangle)'^{(2)}\oplus\langle-2\rangle$\Tstrut} &$e=_4 0$ &$M''_k\oplus U(2) $ &$ e/4+1$\Tstrut  \\
\ &$e=_4 1$ &$N''_k\oplus\langle-4\rangle $ &$(e-1)/4$\Tstrut \\
\ &$e=_4 2$ &$M_k\oplus U(2)\oplus\langle -2\rangle $ &$ (e+2)/4$\Tstrut  \\
\ &$e=_4 3$ &$N_k\oplus\langle-4\rangle\oplus\langle-2\rangle $ &$(e-3)/4$\Tstrut  \\[6pt]
\hline
\end{longtable}
\begin{gather*}G_k=\left[\begin{array}{r r}-4k &2\\ 2&-2\end{array}\right]\quad  S_k=\left[\begin{array}{r r r r}-4k &2 &0 &4\\
     2 &0 &4 &0\\
     0 &4 &0 &0\\
     4 &0 &0 &-4\end{array}\right]\quad S'_k = \left[\begin{array}{ c | c }
    -8 & \begin{matrix} 2 & 0 &0 &0\end{matrix}\\
    \hline
    \begin{matrix} 2 \\ 0\\ 0 \\ 0\end{matrix} & S_k
  \end{array}\right]\\
R_k=\left[\begin{array}{r r r r r r}-4k+4 &2 &0 &0 &2 &0 \\
     2 &0 &0 &0 &2 &0 \\
     0 &0 &-4 &0 &2 &0 \\
     0 &0 &0 &-4 &2 &0 \\
     2 &2 &2 &2 &-4 &2 \\
     0 &0 &0 &0 &2 &-4 \end{array}\right] \quad R'_k = \left[\begin{array}{ c | c }
    -8 & \begin{matrix} 2 & 0 &\dots &0\end{matrix}\\
    \hline
    \begin{matrix} 2 \\ 0\\ \vdots \\ 0\end{matrix} & R_k
  \end{array}\right]\\
Q_k=\left[\begin{array}{r r r r r}8-4k &2 &2 &2 &-2\\
     2 &-4 &0 &0 &0\\
     2 &0 &-4 &0 &0\\
     2 &0 &0 &0 &2\\
     -2 &0 &0 &2 &0\end{array}\right]\ Q'_k = \left[\begin{array}{ c | c }
    -8 & \begin{matrix} 2 & 0 &\dots &0\end{matrix}\\
    \hline
    \begin{matrix} 2 \\ 0\\ \vdots \\ 0\end{matrix} & Q_k
  \end{array}\right]\\ P_k=\left[\begin{array}{r r r r}-4-4k &2 &2 &0\\
     2 &-4 &0 &0\\
     2 &0 &0 &4\\
     0 &0 &4 &0\end{array}\right] \ P'_k = \left[\begin{array}{ c | c }
    -8 & \begin{matrix} 2 & 0 &0 &0\end{matrix}\\
    \hline
    \begin{matrix} 2 \\ 0\\ 0 \\ 0\end{matrix} & P_k
  \end{array}\right]\\
M_k=\left[\begin{array}{r r r r r}8-4k &2 &2 &0 &0\\
     2 &-4 &0 &0 &0\\
     2 &0 &0 &0 &2\\
     0 &0 &0 &-4 &2\\
     0 &0 &2 &2 &-8\end{array}\right]\ 
M'_k = \left[\begin{array}{ c | c }
    -8 & \begin{matrix} 2 & 0 &\dots &0\end{matrix}\\
    \hline
    \begin{matrix} 2 \\ 0\\ \vdots \\ 0\end{matrix} & M_k
  \end{array}\right]\ M''_k = \left[\begin{array}{ c | c }
    -2 & \begin{matrix} 2 & 0 &\dots &0\end{matrix}\\
    \hline
    \begin{matrix} 2 \\ 0\\ \vdots \\ 0\end{matrix} & M_k
  \end{array}\right]\\
N_k=\left[\begin{array}{r r r r r r}-4k &2 &2 &2 &2 &0\\
     2 &-4 &0 &0 &0 &0\\
     2 &0 &0 &2 &0 &0\\
     2 &0 &2 &0 &0 &0\\
     2 &0 &0 &0 &0 &-2\\
     0 &0 &0 &0 &-2 &0\end{array}\right]\ N'_k = \left[\begin{array}{ c | c }
    -8 & \begin{matrix} 2 & 0 &\dots &0\end{matrix}\\
    \hline
    \begin{matrix} 2 \\ 0\\ \vdots \\ 0\end{matrix} & N_k
  \end{array}\right]\ N''_k = \left[\begin{array}{ c | c }
    -2 & \begin{matrix} 2 & 0 &\dots &0\end{matrix}\\
    \hline
    \begin{matrix} 2 \\ 0\\ \vdots \\ 0\end{matrix} & N_k
  \end{array}\right]\\
\end{gather*}
\normalsize
\begin{proof}
The embedding $\Omega_\kappa\hookrightarrow\Lambda_{\mathrm N}$ is always such that $\Omega_\kappa^\perp\simeq\pi_{S*}H^2(S,\mathbb Z)\oplus\langle-2\rangle$, where $S$ is a K3 surface with a symplectic involution $i$, and $\pi_S:S\rightarrow S/i$.  Call $\alpha$ the generator of $(\pi_{S*}H^2(S,\mathbb Z)\oplus\Omega_\kappa)^\perp$, and embed $\Omega_\iota$ in $\Omega_\kappa^{\perp}$ such that the condition $(\star)$ is satisfied.
Denote $\Omega=\Omega_\iota\oplus\Omega_\kappa$. \\
Suppose now that $i=\tau^2\in \langle\tau\rangle=\mathbb Z/4\mathbb Z$, that acts symplectically on $S$. We remark that the gluings in Lemma \ref{incollamenti_involuzioni_indotte}.1 are exactly the image via $\pi_{S*}$ of the ones that define $H^2(S,\mathbb Z)$ as overlattice of finite index of $\Omega_4\oplus \Lambda_{\mathrm{K3}}^\tau$. Therefore, we can obtain all the projective families of $N$ using as generator of $\Omega^{\perp_{NS(N)}}$ an element of the form $\pi_{S*}L+n\alpha$, with $L$ an ample class on $S$ and $n=0,1,2$. 
This bound on $n$ is given by the condition $(\star)$, that allows overlattices of $\Omega\oplus\langle 2d\rangle$ of index at most 2: indeed, the classes of isomorphic overlattices of index 2 vary with the value of $d$ (mod 4).\\ If $n=0$ we find all the projective families with $NS(N)=\pi_{S*}NS(S)\oplus\langle-2\rangle,\ T(N)=\pi_{S*}T(S)\oplus\langle-2\rangle$; if $n=2$, since $(\pi_{S*}L+2\alpha)^2=(\pi_{S*}L)^2-8$, but $\pi_{S*}L+2\alpha$ glues to the same isometry class of $A_{\Omega}$ as $\pi_{S*}L$, we can find new projective families; lastly, if $n=1$ we find all the projective families with $NS(N)=\Omega\oplus\langle2e\rangle$ for $e$ odd.
\end{proof}

\begin{theorem}
Let $N$ be a general projective Nikulin-type orbifold with a symplectic action of $\mathfrak K=\langle\iota,\kappa\rangle$, where $\kappa$ is the non-standard involution described in Theorem \ref{MR_involution}
and $(N,\iota)$ is a deformation of the natural pair $(Y,\tilde\iota)$, where $Y$ is the terminalization of $X/i$, $X$ is a K3$^{[2]}$-type manifold with a symplectic action of $G=(\mathbb Z/2\mathbb Z)^2$ and $i\in G$ of order 2. Then $N$ belongs to one of the projective families described in the following table. 
\end{theorem}
\footnotesize
\begin{longtable}{|c|c|c|c|}
\hline
{$NS(Y)$\Tstrut} & $e$ &{$T(Y)$ \Tstrut} &$k$ \\ [6pt]
\hline
\multirow{18}{*}{$D_4(2)\oplus\langle2e\rangle\oplus\langle-2\rangle$\Tstrut} &\multirow{2}{*}{$e=_2 0$} & $\langle -4k\rangle\oplus\langle -2\rangle\oplus U^{\oplus 2}\oplus D_4(2)$\Tstrut &$e/2$  \\[6pt]
 &\ & $M''_k\oplus U\oplus D_4(2)$ &$e/2$ \\[6pt]
\cline{2-4}
\  &\multirow{4}{*}{$e=_8 1$}& $G_k\oplus U^{\oplus 2}\oplus D_4(2)$ &$(e+1)/2$ \Tstrut  \\[6pt]
\ &\ & $M_k\oplus\langle -2\rangle\oplus U\oplus D_4(2)$ &$(e-1)/2$  \\[6pt]
\ &\ & $R'_k\oplus U$ &$(e-1)/8$  \\[6pt]
\cline{2-4}
\  &\multirow{5}{*}{$e=_8 3$}& $G_k\oplus U^{\oplus 2}\oplus D_4(2)$ &$(e+1)/2$ \Tstrut  \\[6pt]
\ &\ & $M_k\oplus\langle -2\rangle\oplus U\oplus D_4(2)$ &$(e-1)/2$  \\[6pt]
\ &\ & $N'_k\oplus U\oplus U(2)$ &$(e-3)/8$  \\[6pt]
\ &\ & $S'_k\oplus D_4$ &$(e-3)/8$ \\[6pt]
\cline{2-4}
\  &\multirow{4}{*}{$e=_8 5$}& $G_k\oplus U^{\oplus 2}\oplus D_4(2)$ &$(e+1)/2$ \Tstrut  \\[6pt]
\ &\ & $M_k\oplus\langle -2\rangle\oplus U\oplus D_4(2)$ &$(e-1)/2$  \\[6pt]
\ &\ & $Q'_k\oplus U^{\oplus 2}$ &$(e+3)/8$  \\[6pt]
\cline{2-4}
\  &\multirow{2}{*}{$e=_8 7$}& $G_k\oplus U^{\oplus 2}\oplus D_4(2)$ &$(e+1)/2$ \Tstrut  \\[6pt]
\ &\ & $M_k\oplus\langle -2\rangle\oplus U\oplus D_4(2)$ &$(e-1)/2$ \\[6pt]
\hline
\multirow{8}{*}{$(D_4(2)\oplus\langle4e\rangle)'\oplus\langle-2\rangle$\Tstrut} &\multirow{2}{*}{$e=_4 0$} &$P_k\oplus\langle-2\rangle\oplus U$ &$e/4$ \Tstrut  \\[6pt]
\ &\ & $S''_k\oplus D_4$ &$e/4$  \\[6pt]
\cline{2-4}
\ &$e=_4 1$ &$R_k\oplus\langle-2\rangle\oplus U$ &$(e-1)/4$ \Tstrut  \\[6pt]
\cline{2-4}
\ &\multirow{2}{*}{$e=_4 2$} &$N_k\oplus\langle-2\rangle\oplus U\oplus U(2)$ &$(e-2)/4$ \Tstrut  \\[6pt]
\ &\ & $S_k\oplus\langle-2\rangle\oplus D_4$ &$(e-2)/4$  \\[6pt]
\cline{2-4}
\ &$e=_4 3$ &$Q_k\oplus\langle-2\rangle\oplus U^{\oplus 2}$ &$(e+1)/4$ \Tstrut \\[6pt]
\hline
\end{longtable}
\begin{gather*}
G_k=\left[\begin{array}{r r}-4k &2\\ 2&-2\end{array}\right]\quad M_k=\left[\begin{array}{r r r}-4k &2 &0\\
     2 &0 &2\\
     0 &2 &-2\end{array}\right]\quad M''_k = \left[\begin{array}{ c | c }
    -2 & \begin{matrix} 2 & 0 &0\end{matrix}\\
    \hline
    \begin{matrix} 2 \\ 0\\  0\end{matrix} & M_k
  \end{array}\right]\\
N_k=\left[\begin{array}{r r r r r}-4-4k &2 &0 &0 &0\\
     2 &-4 &0 &2 &-2\\
     0 &0 &-4 &2 &2\\
     0 &2 &2 &-4 &0\\
     0 &-2 &2 &0 &-4\end{array}\right]\ N'_k = \left[\begin{array}{ c | c }
    -8 & \begin{matrix} 2 & 0 &\dots &0\end{matrix}\\
    \hline
    \begin{matrix} 2 \\ 0\\ \vdots \\ 0\end{matrix} & N_k
  \end{array}\right]\\
P_k=\left[\begin{array}{r r r r r r r}-4-4k &2 &0 &0 &0 &0 &0\\
     2 &-2 &0 &0 &2 &0 &0\\
     0 &0 &2 &0 &2 &0 &0\\
     0 &0 &0 &-4 &0 &2 &2\\
     0 &2 &2 &0 &-4 &2 &-2\\
     0 &0 &0 &2 &2 &-4 &0\\
     0 &0 &0 &2 &-2 &0 &-4\end{array}\right]\\
Q_k=\left[\begin{array}{r r r r r}-4k &2 &0 &0 &0\\
     2 &-16 &4 &4 &4\\
     0 &4 &-4 &0 &0\\
     0 &4 &0 &-4 &0\\
     0 &4 &0 &0 &-4\end{array}\right]\ Q'_k = \left[\begin{array}{ c | c }
    -8 & \begin{matrix} 2 & 0 &\dots &0\end{matrix}\\
    \hline
    \begin{matrix} 2 \\ 0\\ \vdots \\ 0\end{matrix} & Q_k
  \end{array}\right]\\
R_k=\left[\begin{array}{r r r r r r r}-4k &2 &0 &0 &0 &0 &0\\
     2 &0 &2 &0 &0 &0 &0\\
     0 &2 &-2 &1 &0 &0 &0\\
     0 &0 &1 &-4 &2 &0 &0\\
     0 &0 &0 &2 &-4 &-4 &4\\
     0 &0 &0 &0 &-4 &-8 &4\\
     0 &0 &0 &0 &4 &4 &-8\end{array}\right]\ R'_k = \left[\begin{array}{ c | c }
    -8 & \begin{matrix} 2 & 0 &\dots &0\end{matrix}\\
    \hline
    \begin{matrix} 2 \\ 0\\ \vdots \\ 0\end{matrix} & R_k
  \end{array}\right]\\
S_k=\left[\begin{array}{r r r r r}-4k &2 &0 &0 &0\\
     2 &8 &4 &0 &0\\
     0 &4 &2 &1 &0\\
     0 &0 &1 &-2 &1\\
     0 &0 &0 &1 &-2\end{array}\right]\ S'_k = \left[\begin{array}{ c | c }
    -8 & \begin{matrix} 2 & 0 &\dots &0\end{matrix}\\
    \hline
    \begin{matrix} 2 \\ 0\\ \vdots \\ 0\end{matrix} & S_k
  \end{array}\right]\ S''_k = \left[\begin{array}{ c | c }
    -2 & \begin{matrix} 2 & 0 &\dots &0\end{matrix}\\
    \hline
    \begin{matrix} 2 \\ 0\\ \vdots \\ 0\end{matrix} & S_k
  \end{array}\right]
\end{gather*}
\normalsize
\begin{proof}
The proof is similar to that of Theorem \ref{mixedNikZ4}.
\end{proof}

\end{document}